\documentclass[a4paper,reqno,10pt]{amsart}

\allowdisplaybreaks[1]

\usepackage{geometry}

\usepackage{graphicx}
\usepackage{enumerate}
\usepackage{courier}
\usepackage{times}
\usepackage{amsmath,amssymb}
\usepackage{paralist}

\theoremstyle{plain}
\newtheorem{Theorem}{Theorem}[section]
\newtheorem{Proposition}[Theorem]{Proposition}
\newtheorem{Conjecture}[Theorem]{Conjecture}
\newtheorem{Corollary}[Theorem]{Corollary}
\newtheorem{Lemma}[Theorem]{Lemma}

\theoremstyle{definition}

\theoremstyle{remark}
\newtheorem{Remark}[Theorem]{Remark}

\title{Classification of $\mathbf{(3 \!\mod 5)}$ arcs in $\mathbf{\PG(3,5)}$}

\author[S.~Kurz]{}
\subjclass{Primary: 51E22; Secondary: 51E21, 94B05.}
\keywords{Projective geometries, optimal linear codes, quasi-divisible arcs, $(t\mod q)$-arcs, Griesmer bound.}

 \email{sascha.kurz@uni-bayreuth.de}
 \email{i.landjev@nbu.bg}
 \email{assia@fmi.uni-sofia.bg}

\newcommand{\cB}{\mathcal{B}}
\newcommand{\cF}{\mathcal{F}}

\newcommand{\cH}{\mathcal{H}}
\newcommand{\cK}{\mathcal{K}}

\newcommand{\cP}{\mathcal{P}}
\newcommand{\cQ}{\mathcal{Q}}
\newcommand{\PG}{\operatorname{PG}}
\newcommand{\F}{\mathbb{F}}
\newcommand{\N}{\mathbb{N}}

\begin{document}
\maketitle

\centerline{\scshape Sascha Kurz$^\dagger$, Ivan Landjev$^\ddag$, and Assia Rousseva$^\S$}
\medskip
{\footnotesize
 \centerline{$^\dagger$Mathematisches Institut, Universit\"at Bayreuth, D-95440 Bayreuth, Germany}
   \centerline{$^\ddag$Institute of Mathematics and Informatics, Bulgarian Academy of Sciences, 8 Acad G. Bonchev str., 1113 Sofia, Bulgaria}
   \centerline{$^\ddag$New Bulgarian University, 21 Montevideo str, 1618 Sofia, Bulgaria}
   \centerline{$\S$Sofia University, Faculty of Mathematics and Informatics,J. Bourchier Blvd., 1164 Sofia, Bulgaria}
}

\bigskip

\begin{abstract}
  The proof of the non-existence of Griesmer $[104, 4, 82]_5$-codes is just one of many examples where extendability results are used. 
  In a series of papers Landjev and Rousseva have introduced the concept of $(t\mod q)$-arcs as a general framework for extendability results 
  for codes and arcs. Here we complete the known partial classification of $(3 \mod 5)$-arcs in $\PG(3,5)$ and uncover two 
  missing, rather exceptional, examples disproving a conjecture of Landjev and Rousseva. As also the original non-existence proof of Griesmer 
  $[104, 4, 82]_5$-codes is affected, we present an extended proof to fill this gap.  
\end{abstract}

\section{Introduction}
An $[n,k,d]_q$-code is a $q$-ary linear code with length $n$, dimension $k$, and minimum Hamming distance $d$. Given the field size $q$, a  
main problem in coding theory is to optimize the three remaining parameters. So, let $n_q(k,d)$ denote the minimal length of a linear code 
over $\F_q$ for fixed dimension $k$ and minimum distance $d$. The so-called \emph{Griesmer bound}, see e.g.\ 
\cite{griesmer1960bound,solomon1965algebraically}, is given by 
\begin{equation}
  n_q(k,d)\ge g_q(k,d):=\sum_{i=0}^{k-1} \left\lceil\frac{d}{q^i}\right\rceil.
\end{equation}
Codes attaining this bound are called \emph{Griesmer codes}. In \cite{baumert1973note} is was shown that for all sufficiently large values of $d$, 
depending on $k$ and $q$, we have $n_q(k,d)=g_q(k,d)$. The exact value of $n_q(k,d)$ is known for all $k\le 8$ when $q=2$, for all $k\le 5$ when $q=3$, 
for all $k\le 4$ when $q=4$, and for all $k\le 3$ when $q\le 9$. For $q=5$ and $k=4$ only four cases of $n_5(4,d)$ where unknown before \cite{landjev2016non}. 
Here we fill a gap in the corresponding non-existence proof of Griesmer $[104, 4, 82]_5$-codes. The three remaining unsettled cases are $d\in\{81,161,162\}$. 

In order to show $n_q(k,d)>g_q(k,d)$, the non-existence of an $[g_q(k,d),k,d]_q$-code has to be proven. To this end, so-called extendability results are 
used in many cases. Arguably, the most simple extendability result is that adding a parity check bit to an $[n, k, d]_2$-code with odd minimum distance $d$ yields an
$[n + 1, k, d + 1]_2$-code. In \cite{hill1995extensions,hill1999extension} Hill and Lizak have shown that an $[n, k, d]_q$-code, where $\gcd(d,q)=1$ and the 
weights of the codewords all are either congruent to $0$ or $d$ modulo $q$, is extendable to an $[n + 1, k, d + 1]_q$-code.

In most parts of the paper we will use the geometric description of linear codes as \emph{multisets of points} or \emph{arcs} $\cK$ in the projective geometry 
$\PG(v-1,q)$, see e.g.\ \cite{dodunekov1998codes}. An arc $\cK$ is a \emph{Griesmer arc} if the corresponding code is a Griesmer code. Griesmer codes commonly 
have certain restrictions on their possible weights modulo some divisor $\Delta$ and the corresponding arcs are called quasi-divisible, see e.g.\ \cite{landjev2013extendability}.  
A particularly structured subclass of quasi-divisible arcs are so-called $(t\mod q)$-arcs. The extendability of Griesmer arcs is closely linked to the structure of quasi-divisible 
and $(t\mod q)$-arcs, see e.g.\ \cite{landjev2019divisible,landjev2016extendability}. Partial classification results for $(3 \mod 5)$-arcs in $\PG(3,5)$ were given 
in \cite[Theorem 4.1]{landjev2017characterization} and \cite[Theorem 6]{rousseva2015structure}. Due to a flaw examples with cardinalities $128$ and $143$ where missed, which 
also affects the original non-existence proof of Griesmer $[104, 4, 82]_5$-codes \cite{landjev2016non}. The main target of this article is the full classification of 
all $(3 \mod 5)$-arcs in $\PG(3,5)$. In total there are three examples that do not arise by a lifting construction, see e.g.\ \cite[Theorem 5]{landjev2019divisible}, 
which disproves a conjecture of Landjev and Rousseva.  

The remaining part is structured as follows. In Section~\ref{sec_preliminaries} we present the necessary preliminaries. Known constructions and characterization results 
for $(t\mod q)$-arcs are the topic of Section~\ref{sec_t_mod_q_arcs_constructions}. We also slightly extend the known classification result for strong $(2\mod q)$-arcs in $\PG(2,q)$ 
and give a self-contained proof. We briefly discuss the classification of strong $(3\mod 5)$-arcs in $\PG(2,5)$ in Section~\ref{section_strong_3_mod_5_in_pg_2_5} before we  
treat the classification of all strong $(3\mod 5)$-arcs in $\PG(3,5)$ in Section~\ref{section_strong_3_mod_5_in_pg_3_5}. The adjusted proof of the non-existence of $(104,22)$-arcs 
in $\PG(3,5)$ is the topic of Section~\ref{sec_non-existence_104_22_4_5}. Since this proof relies on several computer calculations we present theoretical substitutes for 
most parts in Subsection~\ref{subsec_shortcuts}. The combinatorial details of all strong $(3\mod 5)$-arcs in $\PG(2,5)$ are presented in an appendix.  

\section{Preliminaries}
\label{sec_preliminaries}
Let $\cP$ denote the set of points and $\cH$ be the set of hyperplanes of $\PG(v-1,q)$, where $v\ge 2$. We have $\#\cP=\#\cH=[v]_q$, where 
$[k]_q:=\left(q^k-1\right)/(q-1)$ for all $k\in\N$. Every mapping $\cK\colon\cP\to\mathbb{N}_0$ is called 
a \emph{multiset} (of points) in $\PG(v-1,q)$. We extend such a mapping additively to subsets $\cQ$ of $\cP$, i.e., $\cK(\cQ)=\sum_{P\in\cQ} \cK(P)$. If $S$ is an $s$-dimensional 
subspace, using the algebraic dimension, we speak of an \emph{$s$-space}. I.e., $1$-spaces are points, $2$-spaces are lines, and $(v-1)$-spaces in $\PG(v-1,q)$ are hyperplanes. 
We also write $\cK(S)$ associating an $s$-space $S$ with the set of its points. The integer $\cK(P)$ is also called the \emph{multiplicity} of a point $P\in\cP$ and $n:=\cK(\cP)$ 
the \emph{cardinality} of $\cK$. For each integer $f$ an \emph{$f$-point} is a point $P$ with multiplicity $\cK(P)=f$, an $f$-line is a line $L$ with multiplicity $\cK(L)=f$, and 
an $f$-hyperplane $H$ is a hyperplane with multiplicity $\cK(H)=f$. The support of $\cK$ is given by 
$\operatorname{supp}(\cK)=\{P\in\cP\,:\, \cK(P)>0\}$. Given a subset $\cQ\subseteq \cP$ the corresponding \emph{characteristic (multi-) set} $\chi_{\cQ}\colon\cP\to\{0,1\}$ 
is given by $\chi_{\cQ}(P)=1$ iff $P\in\cQ$. By $a_i$ we denote the number of hyperplanes $H\in\cH$ with $\cK(H)=i$ and call the sequence $(a_i)_{i\in\N_0}$ the  
\emph{spectrum} of $\cK$. By double-counting incidences between points and hyperplanes one obtains the so-called \emph{standard equations}:
\begin{eqnarray}
    \sum_{i\ge 0} a_i &=& [v]_q \label{eq_standard_equation_1}\\ 
    \sum_{i\ge 0} ia_i &=& n\cdot [v-1]_q\label{eq_standard_equation_2}\\ 
    \sum_{i\ge 0} {i \choose 2}a_i &=& {n\choose 2}\cdot[v-2]_q+q^{v-2}\cdot\sum_{i\ge 2} {i\choose 2}\lambda_i\label{eq_standard_equation_3},
\end{eqnarray}
where $\lambda_j$ denotes the number of points $P\in\cP$ with $\cK(P)=j$ for all $j\in \N_0$. The coding theoretic analog of the standard equations are the first three 
MacWilliams equations. For the $\lambda_i$ we have
\begin{equation}
  \sum_{i\ge 0} \lambda_i \,=\, [v]_q\quad\text{and}\quad \sum_{i\ge 0} i\lambda_i \,=\, n.
\end{equation}

If a multiset $\cK$ has cardinality $n$ and satisfies $\cK(H)\le s$ for all hyperplanes $H\in\cH$, then we call $\cK$ an \emph{$(n,\le s)$-arc} and an \emph{$(n,s)$-arc}  
if additionally a hyperplane $H$ with $\cK(H)=s$ exists. Similarly, a multiset $\cK$ with cardinality $n$ and $\cK(H)\ge s$ for all $H\in\cH$ is called an 
\emph{$(n\ge s)$-blocking set with respect to hyperplanes} or an \emph{$(n,\ge s)$-minihyper}. If a hyperplane $H\in\cH$ with $\cK(H)=s$ exists, then we write $(n,s)$ instead 
of $(n,\ge s)$. Specifying the mentioned relation between linear codes and arcs, we state that there exists a one-to-one correspondence between the classes of isomorphic 
$[n,k,d]_q$-codes and the classes of projectively equivalent $(n,n-d)$-arcs in $\PG(k-1,q)$.  

An $(n,s)$-arc $\cK$ in $\PG(v-1,q)$ is called \emph{$t$-extendable} if there exists an $(n+t,s)$-arc $\cK'$ in $\PG(v-1,q)$ with $\cK'(P)\ge \cK(P)$ for all $P\in\cP$. 
If $\cK$ is $t$-extendable for some $t\ge 1$, we also say that $\cK$ is \emph{extendable}. Similarly, an $(n,s)$-minihyper $\cK$ in $\PG(v-1,q)$ is called \emph{reducible}, 
if there exists an $(n-1,s)$-minihyper $\cK'$ in $\PG(v-1,q)$ with $\cK'(P)\le \cK(P)$ for all $P\in\cP$. A minihyper that is not reducible is called \emph{irreducible}.

An $(n,s)$-arc $\cK$ with spectrum $(a_i)$ is \emph{divisible with divisor $\Delta$} if $a_i=0$ for all $i \neq  n \pmod \Delta$ and $\Delta>1$. For the corresponding linear 
code the condition says that the weights of all codewords are divisible by $\Delta$. More generally, an $(n,s)$-arc $\cK$ with $s\equiv n+t\pmod \Delta$ is called 
\emph{$t$-quasi-divisible with divisor $\Delta$} if $a_i=0$ for all $i \not\equiv n,n+1,\dots, n+t \pmod \Delta$ and $1\le t\le q-1$. An arc $\cK$ in $\PG(v-1,q)$ 
is called a \emph{$(t\mod q)$-arc}, where $1\le t\le q-1$, if $\cK(L)\equiv t\pmod q$ for every line $L$. By double-counting one easily sees that also $\cK(S)\equiv t\pmod q$  
is satisfied for every subspace $S$ of larger dimension. If the maximum point multiplicity of $\cK$ is at most $t$, i.e., $\cK(P)\le t$ for all $P\in\cP$, then we call 
$\cK$ a \emph{strong $(t\mod q)$-arc} noting that some papers use the notion of $(t \mod q)$-arcs for strong $(t \mod q)$-arcs. Note that increasing the point multiplicities 
of arbitrary points by multiples of $q$ preserves the property of being a $(t\mod q)$-arc.

Let $\cK$ be an arc and $\sigma\colon\N_0\to\mathbb{R}$ be a function satisfying $\sigma(\cK(H))\in\N_0$ for every hyperplane $H\in\cH$. The arc 
$\cK^\sigma\colon\cH\to\N_0$, $H\mapsto\sigma(\cK(H))$ is called the \emph{$\sigma$-dual} of $\cK$. The roles of points and hyperplanes have to be interchanged.  
Note that taking $\sigma$ as the identity function on $\N_0$ gives the dual arc $\cK^\perp$. If $\sigma$ is linear, then the parameters of $\cK^\sigma$ can be easily 
computed from the parameters of $\cK$, see e.g.\ \cite{brouwer1997correspondence}. For a $t$-quasi-divisible arc a special $\sigma$-dual arc is of importance. Let $\cK$ 
be a $t$-quasi-divisible $(n,s)$-arc with divisor $q$ in $\PG(v-1,q)$, where $1\le t<q$. By $\widetilde{\cK}$ we denote the $\sigma$-dual of $\cK$ in the dual geometry 
$\PG^\perp(k-1,q)$, where $\sigma(x)=n+t-x \mod q$. More precisely, we have $\widetilde{\cK}\colon \cH\to\{0,1,\dots,t\}$, 
\begin{equation}
  \label{eq_quasidivisible_dual}
  H\mapsto \widetilde{\cK}(H)\equiv n+t-\cK(H) \pmod q.
\end{equation}  
In other words, hyperplanes of multiplicity congruent to $n+a \pmod q$ become $(t-a)$-points in the dual geometry. In particular, $s$-hyperplanes become $0$-points with 
respect to $\widetilde{\cK}$. In general, the cardinality of $\widetilde{\cK}$ cannot be obtained from the parameters of $\cK$.

Defining the sum of two multisets $\cK$ and $\cK'$ in the same geometry by $(\cK+\cK')(P)=\cK(P)+\cK'(P)$ for all $P\in \cP$, the following
theorem is straightforward.
\begin{Theorem}(E.g.\ \cite[Theorem 1]{landjev2016extendability})
  \label{thm_quasidivisible_extendability}
  Let $\cK$ be an $(n,s)$-arc in $\PG(v-1,q)$, which is $t$-quasi-divisible with divisor $q$, where $1\le t<q$. Let $\widetilde{\cK}$ defined 
  by Equation~(\ref{eq_quasidivisible_dual}). If 
  \begin{equation}
    \widetilde{\cK}= \sum_{i=1}^c \chi_{\widetilde{P}_i}+\cK'
  \end{equation}
  for some multiset $\cK'$ in the dual geometry $\PG^\perp(v-1,q)$ and $c$ not necessarily different hyperplanes $\widetilde{P}_1,\dots,\widetilde{P}_c$ 
  in $\PG^\perp(k-1,q)$, then $\cK$ is $c$-extendable. In particular, if $\widetilde{\cK}$ contains a hyperplane in its support, then 
  $\cK$ is extendable.
\end{Theorem}
Let us note that the condition of Theorem~\ref{thm_quasidivisible_extendability} is sufficient, but not necessary, since 
$0$-points in $\PG^\perp(v-1,q)$ with respect to $\widetilde{\cK}$ can correspond to hyperplanes in $\PG(v-1,q)$ that are not of 
the maximum possible multiplicity with respect to the $(n,s)$-arc $\cK$. However, in some situations $\cK(H)\equiv s\pmod q$, 
where $H\in\cH$, implies $\cK(H)=s$.

\begin{Theorem} (E.g.~\cite[Theorem 2]{landjev2016extendability})  
  \label{Theorem_quasidivisile_implies_highly_divisible}
  Let $\cK$ be an $(n,s)$-arc in $\PG(v-1,q)$ which is $t$-quasi-divisible with divisor $q$, where $1\le t<q$. For every line $\widetilde{L}$, 
  in the dual geometry $\PG^\perp(v-1,q)$ we have
  \begin{equation}
    \widetilde{\cK}\!\left(\widetilde{L}\right)\equiv t\pmod q.
  \end{equation}
\end{Theorem}  
In other words, $\widetilde{\cK}$ is a strong $(t \mod q)$-arc, c.f.~\cite[Corollary 1]{landjev2016extendability}, and Theorem~\ref{thm_quasidivisible_extendability} 
links the extendability problem to the classification problem of strong $(t\mod q)$-arcs. Note that this correspondence is not injective, i.e., different non-isomorphic 
$t$-quasi-divisible arcs can produce the same strong $(t\mod q)$-arc. The mapping $^\sim$ is also not surjective since strong $(t\mod q)$-arcs without $0$-points and 
$1\le t<q$ cannot be obtained by (\ref{eq_quasidivisible_dual}) from $t$-quasi-divisible arcs. However, it is not clear whether all strong $(t\mod q)$-arcs with $0$-points 
and $1\le t<q$ come from $t$-quasi-divisible arcs. In the remaining part of the paper we want to study $(t \mod q)$-arcs as purely geometric objects without using the relation 
to the extendability problem.

\section[Constructions and characterizations for $(t\mod q)$-arcs]{Constructions and characterizations for $\mathbf{(t\mod q)}$-arcs}
\label{sec_t_mod_q_arcs_constructions}
A few constructions for $(t \mod q)$-arcs are known. First we mention that two such arcs can be combined to a $(t \mod q)$-arc with a larger value of $t$. 
\begin{Theorem}(\cite[Theorem 4]{landjev2019divisible})
  \label{thm_sum_of_t_mod_q}
  Let $\cK$ and $\cK'$ be a $(t_1 \mod q)$- and a $(t_2 \mod q)$-arc in $\PG(v-1, q)$, respectively. Then $\cK+\cK'$ is a $(t \mod q)$-arc with 
  $t \equiv t_1 + t_2 \pmod q$. Similarly, $\alpha \cK$, where $\alpha\in\{ 0,\dots,\dots, p-1\}$ and $p$ is the characteristic of $\F_q$, is a $(t \mod q)$-arc 
  with $t\equiv \alpha t_1\pmod q$.
\end{Theorem}
When $t = 0$ and $q=p$ is a prime then Theorem~\ref{thm_sum_of_t_mod_q} directly implies the following nice characterization.
\begin{Corollary}(\cite[Corollary 1]{landjev2019divisible})
  Let $\cK$ and $\cK'$ be $(0 \mod p)$-arcs in $\PG(v-1, p)$, where $p$ is a prime. Then $\cK+\cK'$ and and $\alpha\cK$, where $\alpha\in\{0,\dots,p-1\}$, are also 
  $(0 \mod p)$-arcs. In particular, the set of all $(0 \mod p)$-arcs in $PG(v-1, p)$ is a vector space over $\F_p$.
\end{Corollary}
In \cite[Theorem 7]{landjev2019divisible} the authors show that the vector space of all $(0\mod p)$-arcs in $\PG(v-1,p)$ is generated by the complements of hyperplanes. Of course 
the only strong $(0\mod q)$-arc $\cK$ is the empty arc with $\#\cK=0$. It is an easy exercise to show that strong $(1\mod q)$-arcs in $\PG(v-1,q)$ are either given by 
$\chi_{\cP}$ or $\chi_H$, where $H\in\cH$ is an arbitrary hyperplane.

The so-called \emph{lifting construction} is given by:
\begin{Theorem}(\cite[Theorem 2]{rousseva2015structure})
\label{lifted_arcs_construction}
  Let $\cK_0$ be a $(t\mod q)$-arc in a hyperplane $H\cong \PG(v-2,q)$ of $\PG(v-1,q)$, where $v\ge 2$. For a fixed point 
  $P$ in $\PG(v-1,q)$, not incident with $H$, we define an arc $\cK$ in $\PG(v-1,q)$ as follows: 
  \begin{itemize}
    \item $\cK(P)=t$;
    \item for each point $Q\neq P$ in $\cP$ we set $\cK(Q)=\cK_0(R)$, where $R=\langle P,Q\rangle\cap H$. 
  \end{itemize}  
  Then, $\cK$ is a $(t\mod q)$-arc in $\PG(v-1,q)$ of cardinality $q\cdot \#\cK_0+t$. If $\cK_0$ is strong, so is $\cK$.
\end{Theorem}
We call the $(t\mod q)$-arcs obtained from Theorem~\ref{lifted_arcs_construction} \emph{lifted arcs} and the point $P$ the 
\emph{lifting point}. It is possible that a lifted arc can be obtained from several different lifting points.
\begin{Lemma}(\cite[Lemma 1]{landjev2019divisible})
  \label{lemma_lifting_points_subspace}
  Let $\cK$ be a lifted arc. If $P,Q$ are lifting points for $\cK$, then any point on the line $\langle P,Q\rangle$ is a lifting point. In particular, the 
  lifting points of $\cK$ form a subspace.
\end{Lemma}
The characteristic function of a hyperplane is indeed a lifted arc and for quite some time the only known strong $(t \mod q)$-arcs in $\PG(v-1,q)$, where $v\ge 4$, 
were lifted arcs. We present three non-lifted $(3 \mod 5)$-arcs in $\PG(3,5)$ in Section~\ref{section_strong_3_mod_5_in_pg_3_5}.
\begin{Theorem}(\cite[Theorem 9]{landjev2019divisible})
  \label{thm_everything_is_lifted}
  Let $\cK$ be a $(t \mod q)$-arc in $\PG(v-1,q)$ such that the restriction $\cK|_H$ to every hyperplane $H\in\cH$ is lifted. Then $\cK$ itself is a lifted arc.
\end{Theorem}

In $\PG(1,q)$ $(t\mod q)$-arcs $\cK$ are very numerous and have little structure, i.e., the only condition is $n\equiv t\pmod q$. For strong $(t\mod q)$-arcs  
in $\PG(1,q)$ additionally the maximum point multiplicity is upper bounded by $t$. For strong $(t\mod q)$-arcs in $\PG(2,q)$ we have the following characterization. 
\begin{Theorem}(\cite[Theorem 10]{landjev2019divisible})
  \label{thm_pg_2_q}
  A strong $(t\mod q)$-arc $\cK$ in $\PG(2,q)$ of cardinality $mq+t$ exists if and only if there exists an $((m-t)q+m,\ge m-t)$-minihyper $\cB$ 
  with line multiplicities contained in $\{m-t,m-t+1,\dots,m\}$.
\end{Theorem}  
For strong $(2\mod q)$-arcs in $\PG(2,q)$ we can say a bit more.
\begin{Lemma}
  \label{lemma_2mod5_arcs_nonnegative_solutions}
  Let $\cK$ be a $q$-divisible arc in $\PG(2,q)$ whose cardinality $n$ is congruent to $2$ modulo $q$ and whose maximum point multiplicity is at most $2$. 
  For $q\ge 5$ we have one of the following possibilities:
  \begin{enumerate}
    \item[(1)] $n=2q+2$, $a_{2}=q^2+q-1$, $a_{q+2}=2$, $a_{2q+2}=0$, $\lambda_0=q(q-1)$, $\lambda_1=2q$, $\lambda_2=1$;
    \item[(2)] $n=2q+2$, $a_{2}=q(q+1)$, $a_{q+2}=0$, $a_{2q+2}=1$, $\lambda_0=q^2$, $\lambda_1=0$, $\lambda_2=q+1$;
    \item[(3)] $n=q^2+q+2$, $a_{2}=q$, $a_{q+2}=q^2+1$, $a_{2q+2}=0$, $\lambda_0=q(q-1)/2$, $\lambda_1=2q$, $\lambda_2=1+q(q-1)/2$;
    \item[(4)] $n=q^2+q+2$, $a_{2}=q+1$, $a_{q+2}=q^2-1$, $a_{2q+2}=1$, $\lambda_0=q(q+1)/2$, $\lambda_1=0$, $\lambda_2=1+q(q+1)/2$;
    \item[(5)] $n=q^2+2q+2$, $a_{2}=i$, $a_{q+2}=q^2+q-2i$, $a_{2q+2}=i+1$, $\lambda_0=iq$, $\lambda_1=q^2 -2iq$, $\lambda_2=1+q(i+1)$, 
               where $0\le i\le \left\lfloor\tfrac{q}{2}\right\rfloor$;
    \item[(6)] $n=(q+1)(q+2)$, $a_{2}=0$, $a_{q+2}=q^2-1$, $a_{2q+2}=q+2$, $\lambda_0=q(q-1)/2$, $\lambda_1=0$, $\lambda_2=(q+1)(q+2)/2$;
    \item[(7)] $n=2\left(q^2+q+1\right)$, $a_{2}=0$, $a_{q+2}=0$, $a_{2q+2}=q^2+q+1$, $\lambda_0=0$, $\lambda_1=0$, $\lambda_2=q^2+q+1$.
  \end{enumerate}  
\end{Lemma}
\begin{proof}
  Solving the standard equations for the hyperplanes and points 
  \begin{eqnarray*}
    a_2+a_{q+2}+a_{2q+2} &=& q^2+q+1 \\ 
    2a_2+(q+2)a_{q+2}+(2q+2)a_{2q+2} &=& n(q+1)\\ 
    a_2+\frac{(q+2)(q+1)}{2}a_{q+2}+(q+1)(2q+1)a_{2q+2} &=& {n\choose 2} +q\lambda_2\\ 
    \lambda_0+\lambda_1+\lambda_2 &=& q^2+q+1\\ 
    \lambda_1+2\lambda_2 &=& n
  \end{eqnarray*}
  for $\left\{a_2,a_{q+2},\lambda_0,\lambda_1,\lambda_2\right\}$ gives
  \begin{eqnarray*}
    a_2 &=& \frac{q^3+x q-n q+3 q^2-n+3 q+2}{q}\\ 
    a_{q+2} &=& -\frac{2 x q-n q+2 q^2-n+2 q+2}{q},\\ 
    \lambda_0 &=& \frac{2xq^2 + nq^2 - n^2 + 2nq - 4q^2 + 4n - 4q - 4}{2q}\\ 
    \lambda_1 &=& -\frac{2xq^2 + nq^2 - 2q^3 - n^2 + 3nq - 6q^2 + 4n - 6q - 4}{q}\\ 
    \lambda_2 &=& \frac{2xq^2 + nq^2 - 2q^3 - n^2 + 4nq - 6q^2 + 4n - 6q - 4}{2q},
  \end{eqnarray*}  
  where we set $x:=a_{2q+2}$ as an abbreviation.
  
  First we treat a few special cases separately. If $n=2q+2$, then the above equations simplify to $a_2 = q^2 + x + q - 1$, $a_{q+2} = -2x + 2$, $\lambda_0 = xq + q^2 - q$, 
  $\lambda_1 =2q(1-x)$, and $\lambda_2 = xq + 1$. From $\lambda_1\ge 0$ and $x\in\N_0$ we conclude $x\in\{0,1\}$, which gives the cases (1) and (2). 
  In the following we assume $n\neq 2q+2$.  
  
  If $n=q^2+q+2$, then the above equations simplify to $a_2 = x + q$, $a_{q+2} = q^2 - 2x + 1$, $\lambda_0 = xq + \tfrac{1}{2}q(q-1)$, $\lambda_1 = 2q(1-x)$, and 
  $\lambda_2 = 1+xq + \tfrac{1}{2}q(q-1)$. From $\lambda_1\ge 0$ and $x\in\N_0$ we conclude $x\in\{0,1\}$, which gives the cases (3) and (4).  
  In the following we assume $n\neq q^2+q+2$. 
  
  If $n=q^2+2q+2$, then the above equations simplify to $a_2 = x - 1$, $a_{q+2} = q^2 - 2x + q + 2$, $\lambda_0 = (x-1)q$, $\lambda_1=  q\cdot\left(q+2 -2x\right)$ , 
  and $\lambda_2 = xq + 1$. From $\lambda_0\ge 0$, $\lambda_1\ge 0$, and $x\in\N_0$ we conclude $x\in\left\{1,2,\dots,1+\left\lfloor\tfrac{q}{2}\right\rfloor\right\}$, 
  which gives the parametric case (5), where $i=x-1$. In the following we assume $n\neq q^2+2q+2$.

  If $n=q^2+3q+2$, then the above equations simplify to $a_2 = x - q - 2$, $a_{q+2} = q^2 - 2x + 2q + 3$, $\lambda_0 = xq - \tfrac{1}{2}q(q+5)$, 
  $\lambda_1 = 2q(q + 2-x)$, and $\lambda_2 = xq - \tfrac{1}{2}q(q+1)+ 1$. From $a_2\ge 0$ and $\lambda_1\ge 0$ we conclude $x=q+2$, which gives case (6). 
  In the following we assume $n\neq q^2+3q+2$. 
  
  Now we are ready to analyze the general situation. From $a_{q+2}\ge 0$ we conclude 
  $$
    n\ge \frac{2(xq + q^2 + q + 1)}{q + 1}\ge \frac{2(q^2+q+1)}{q+1}>2q,
  $$  
  so that $n\equiv 2\pmod q$ and $n\neq 2+2q$ implies $n\ge 2+3q$. The non-negativity of $\lambda_1$ gives
  $$
    n\le 2 +q\cdot\frac{q+3 - \sqrt{q^2\!-\!2q\!-\!7\!+\!8x}}{2} =2 +q\cdot\frac{q+3 - \sqrt{(q\!-\!3)^2\!+\!4(q\!-\!4)\!+\!8x}}{2} 
    \,\,\underset{q\ge 5}{\overset{x\ge 0}{<}}\,\, 3q+2  
  $$
  or
  \begin{equation}
    \label{ie_2mod5_param_range1}
    n\ge 2 +q\cdot\frac{q+3 + \sqrt{q^2 - 2q - 7+8x}}{2}=\frac{q^2+3q +4 +q\sqrt{q^2 - 2q - 7+8x}}{2}    
  \end{equation}
  where we only need to consider Inequality~(\ref{ie_2mod5_param_range1}), due to $n\ge 2+3q$. From $n\equiv 2\pmod q$, the 
  estimation
  $$
    2 +q\cdot\frac{q+3 + \sqrt{q^2 - 2q - 7+8x}}{2}=2 +q\cdot\frac{q+3 + \sqrt{(q\!-\!3)^2\!+\!4(q\!-\!4)\!+\!8x}}{2} 
    \,\,\underset{q\ge 5}{\overset{x\ge 0}{>}}\,\, q^2+2,
  $$ 
  and $n\notin\left\{q^2+q+2,q^2+2q+2,q^2+3q+2\right\}$ we conclude $n\ge q^2+4q+2$. From $a_2\ge 0$ we conclude
  \begin{equation}
    \label{ie_2mod5_param_range2}
    n\le \frac{q^3 + xq + 3q^2 + 3q + 2}{q + 1},
  \end{equation}  
  so that Inequality~(\ref{ie_2mod5_param_range1}) yields
  $$
    \frac{q^3 + xq + 3q^2 + 3q + 2}{q + 1} \ge \frac{q^2+3q +4 +q\sqrt{q^2 - 2q - 7+8x}}{2},
  $$
  which implies $x\le q + 2$ or $x\ge q^2 + q + 1$. If $x\ge q^2+q+1$, then $a_2+a_{q+2}+x=q^2+q+1$, and $a_2,a_{q+2}\ge 0$ imply 
  $a_2=0$, $a_{q+2}=0$, and $x=q^2+q+1$, so that $n=2(q^2+q+1)$. This is case (7). If $x\le q+2$, then Inequality~(\ref{ie_2mod5_param_range2}) 
  implies 
  $$
    n\le \frac{q^3 + q(q+2) + 3q^2 + 3q + 2}{q + 1}=q^2 + 3q + 2, 
  $$
  a range for $n$ that has been treated before.
\end{proof}

We remark that the cases $q\in\{2,3,4\}$ admit the same solutions of the standard equations and a few more:
\begin{eqnarray*}
  n=11, a_2=7, a_5=6, a_8=0, \lambda_0=6, \lambda_1=3, \lambda_2=4
\end{eqnarray*}
for $q=3$ and
\begin{eqnarray*}
   n=14, a_2=14, a_6=7, a_{10}=0, \lambda_0=14, \lambda_1=0, \lambda_2=7\\
   n=18, a_2=9, a_6=12, a_{10}=0, \lambda_0=12, \lambda_1=0, \lambda_2=9
\end{eqnarray*}
for $q=4$.
For $q=3$ the arc can be described as follows. The four $2$-points form an oval, all internal points are $1$-points, and all external points are $0$-points. 
A generator matrix of the corresponding code is e.g.\ given by  
$$
\begin{pmatrix}
  11111111100\\
  00001111210\\
  00110011201\\
\end{pmatrix}.
$$
For $q=4$ we can construct a corresponding projective $2$-divisible arc via $\cK'(P)=\cK(P)/2$ for all $P\in\cP$. The corresponding codes are $2$-weight codes and 
examples are given by the parametric families RT1 and RT3 in \cite{calderbank1986geometry}, respectively.


\begin{Proposition}
  \label{prop_strong_2_mod_q_odd_q}
  Let $q\ge 5$ be odd. For a strong $(2\mod q)$-arc $\cK$ in $\PG(2,q)$ we have the following possibilities:
  \begin{enumerate}
  \item[(I)] A lifted arc from a $2$-line with $\#\cK=2q+2$. There exist two possibilities:
  \begin{enumerate}
    \item[(I-1)] a double line; or
    \item[(I-2)] a sum of two different lines.
  \end{enumerate}  
  \item[(II)] A lifted arc from a $(q+2)$-line $L$ with $\#\cK=q^2+2q+2$ points. The line $L$ has $i$ double points, $q-2i+2$ single points, 
              and $i-1$ $0$-points, where $1\le i\le\tfrac{q+1}{2}$. We say that such an arc is of type (II-i) if it is lifted from a line with $i$ 
              double points.
  \item[(III)] A lifted arc from a $(2q+2)$-line, which is the same as two copies of the plane. Such an arc has $2(q^2+q+1)$ points.
  \item[(IV)] An exceptional $(2\mod q)$-arc for $q$ odd. It consists of the points of an oval, a fixed tangent to this oval, and two copies of each 
              internal point of the oval.                  
  \end{enumerate}
\end{Proposition}
\begin{proof}  
  We apply Lemma~\ref{lemma_2mod5_arcs_nonnegative_solutions} and first note that the cases (4) and (6) cannot occur for odd field sizes $q$ since $\lambda_1=0$ but 
  $a_{q+2}>0$. First observe that each $(2q+2)$-line is a double line, i.e.\ each of the $q+1$ points has multiplicity $2$, and each $(q+2)$-line contains at least 
  one $2$-point. For case (1) there is a unique $2$-point which has to be contained on the two $(q+2)$-lines, so that the remaining $2q$ points on these two lines 
  are $1$-points. This is case (I-2) in the classification. For the case (2) the unique $(2q+2)$-line, $\lambda_1=0$, and $\lambda_2=q+1$ imply case (I-1). In case (7) 
  all points have multiplicity $2$, which corresponds to case (III). For case (5) let us first observe that there are no $0$-points for $i=0$, i.e., setting 
  $\cK'(P)=\cK(P)-1$ for all $P\in\cP$ gives a strong $(1\mod q)$-arc of cardinality $q+1$, which is the characteristic function of a line.  
  The multiset of points given by $\PG(2,q)$ and a line can also be described as in (II-1). For $i\ge 1$ there exist $0$-points, so that the distribution 
  of the multiplicities of the lines through a $0$-point is given by $2^1 (q+2)^q$. Due to the existence of a $(2q+2)$-line, the $2$-line through a $0$-point contains a 
  $2$-point. If $i=1$ there is a unique such $2$-point $Q$. For $i>1$ we observe that all such $2$-lines through $0$-points have to intersect in the same $2$-point (that we 
  also call $Q$). So, through the $2$-point $Q$ there are exactly $i$ two-lines, so that counting points give that the remaining lines through $Q$ split into $i+1$ lines 
  of multiplicity $2q+2$, which contain all $2$-points, and $q-2i$ lines of multiplicity $q+2$, which then consist of $q$ one-points and $Q$. This is the situation described 
  in case (II-$(i+1)$). For the remaining case (3) we consider the dual arc $\cK^\sigma$ with respect to $\sigma(x)=\tfrac{q+2-x}{q}$. With this, $\cK^\sigma$ is a 
  (projective) $(q,2)$-arc in $\PG(2,q)$ which is extendable. An extension point of $\cK^\sigma$ corresponds to a full line in $\cK$. After extending $\cK^\sigma$ we obtain 
  an oval, which yields the description for $\cK$ given in (IV).            
\end{proof}

From the coding theory perspective, Proposition~\ref{prop_strong_2_mod_q_odd_q} was implicitly proven in \cite{maruta2004new}. Having the implications of the standard equations, 
i.e., Lemma~\ref{lemma_2mod5_arcs_nonnegative_solutions} at hand, we can say a bit more. For even field sizes $q$ the case (4) in Lemma~\ref{lemma_2mod5_arcs_nonnegative_solutions} 
can be attained. Removing the unique double line from $\cK$ and halving all point multiplicities yields a projective $q/2$-divisible arc $\cK'$ with cardinality $q(q-1)/2$ and line 
multiplicities $0$ and $q/2$ in $\PG(2,q)$. A corresponding $2$-weight code is contained in the family TF2 in \cite{calderbank1986geometry}. In case (6) halving the point 
multiplicities yields a projective $q/2$-divisible arc $\cK'$ with cardinality $(q+1)(q+2)/2$ and line multiplicities $q(q+1)/2$ and $q(q+2)/2$ in $\PG(2,q)$. 
Corresponding $2$-weight codes are contained in the families TF1d and TF2d in \cite{calderbank1986geometry}.

The implicit classification result of strong $(2 \mod q)$-arcs in $\PG(2,q)$ for odd $q$ from \cite{maruta2004new}, i.e., Proposition~\ref{prop_strong_2_mod_q_odd_q}, 
was used in \cite{landjev2019divisible} to show:
\begin{Theorem}(\cite[Theorem 11]{landjev2019divisible},\cite[Theorem 5]{rousseva2015structure})
  \label{thm_hyperplane_in_support_special}
  Let $\cK$ be a strong $(2\mod q)$-arc in $\PG(v-1,q)$, where $v\ge 4$ and $q$ is odd. Then, $\cK$ is a lifted arc. In particular, for $v\ge 3$ and odd field sizes $q$ every 
  $(2\mod q)$-arc in $\PG(v-1,q)$ has a hyperplane in its support.   
\end{Theorem}
In \cite[Remark 2]{landjev2019divisible} it was mentioned that Theorem~\ref{thm_hyperplane_in_support_special} provides an alternative proof of Maruta’s theorem \cite{maruta2004new} 
on the extendability of linear codes with weights $-2, -1, 0 \pmod q$ over $\F_q$. 

For strong $(t\mod q)$-arcs the situation is far more complicated if $t\ge 3$. E.g.\ for strong $(3\mod q)$-arcs in $\PG(2,q)$ we have 
many strong $(3 \mod q)$-arcs obtained as the sum of a strong $(2 \mod q)$- and a strong $(1\mod q)$-arc, but also some non-trivial indecomposable arcs, see 
Section~\ref{section_strong_3_mod_5_in_pg_2_5} where we fully classify all strong $(3 \mod 5)$-arcs in $\PG(2,5)$. For the sake of completeness we state some easy general observations.
\begin{Lemma}
  Let $\cK$ be a $(t \mod q)$-arc in $\PG(v-1,q)$, where $v\ge 3$. For every hyperplane $H$ in $\PG(v-1,q)$ the restricted arc $\cK|_H$ is a 
  $(t\mod q)$-arc in $\PG(v-2,q)$. If $\cK$ is strong, so is $\cK|_H$.
\end{Lemma}
\begin{Lemma}
  \label{lemma_t_mod_q_arc_sum_of_hyperplanes}
  Let $\cK$ be a $(t \mod q)$-arc in $\PG(v-1,q)$, where $v\ge 2$. Then, we have $\#\cK\ge [v-1]_q\cdot t$.
\end{Lemma}
\begin{Lemma}
  \label{lemma_t_mod_q_hyperplane_in_support}
  If $\cK$ is a $(t \mod q)$-arc in $\PG(v-1,q)$ whose support contains a hyperplane $H$, where $t\ge 1$, then $\cK'=\cK-\chi_H$ is 
  a  $(t-1 \mod q)$-arc in $\PG(v-1,q)$.
\end{Lemma}
Note that it may happen that $\cK$ is strong while $\cK'$ is not strong. This is e.g.\ the case when $\cK$ has full support.
\begin{Proposition}
  \label{prop_t_mod_q_arc_sum_of_hyperplanes}  
  For $v\ge 2$, each $(t \mod q)$-arc $\cK$ in $\PG(v-1,q)$ of cardinality $[v-1]_q\cdot t$ is a sum of $t$ hyperplanes.
\end{Proposition}

\section{Strong $(3\mod 5)$-arcs in $\PG(2,5)$}
\label{section_strong_3_mod_5_in_pg_2_5}
For the exhaustive classification of strong $(3\mod 5)$-arcs in $\PG(2,5)$ we utilize Theorem~\ref{thm_pg_2_q} and generate the corresponding minihypers 
as linear codes using the software package \texttt{LinCode} \cite{bouyukliev2020computer}. In Table~\ref{tab_isomorphism_types_of_3_mod_5_arcs_pg_2_5} 
we list the number of isomorphism types. 
\begin{table}[htp]
\begin{center}
  \begin{tabular}{|c|c|c|c|c|r|}
    \hline
    $\#\cK$ & $m$ & $\#\cB$ & line mult. & weights & \# isomorphism types \\ 
    \hline
    \hline
    18 &  3 &  3 & $0,1,2,3$     & $0,1,2,3$     &  4 \\ 
    23 &  4 &  9 & $1,2,3,4$     & $5,6,7,8$     &  1 \\
    28 &  5 & 15 & $2,3,4,5$     & $10,11,12,13$ &  1 \\
    33 &  6 & 21 & $3,4,5,6$     & $15,16,17,18$ & 10 \\
    38 &  7 & 27 & $4,5,6,7$     & $20,21,22,23$ & 23 \\
    43 &  8 & 33 & $5,6,7,8$     & $25,26,27,28$ & 53 \\
    48 &  9 & 39 & $6,7,8,9$     & $30,31,32,33$ & 49 \\
    53 & 10 & 45 & $7,8,9,10$    & $35,36,37,38$ & 17 \\
    58 & 11 & 51 & $8,9,10,11$   & $40,41,42,43$ & 11 \\
    63 & 12 & 57 & $9,10,11,12$  & $45,46,47,48$ &  9 \\
    68 & 13 & 63 & $10,11,12,13$ & $50,51,52,53$ &  6 \\
    73 & 14 & 69 & $11,12,13,14$ & $55,56,57,58$ &  0 \\
    78 & 15 & 75 & $12,13,14,15$ & $60,61,62,63$ &  0 \\
    83 & 16 & 81 & $13,14,15,16$ & $65,66,67,68$ &  0 \\
    88 & 17 & 87 & $14,15,16,17$ & $70,71,72,73$ &  0 \\
    93 & 18 & 93 & $15,16,17,18$ & $75,76,77,78$ &  1 \\
    \hline
  \end{tabular}
  \caption{Number of isomorphism types of strong $(3\mod 5)$-arcs in $\PG(2,5)$ and their corresponding minihypers.}
  \label{tab_isomorphism_types_of_3_mod_5_arcs_pg_2_5}
\end{center}
\end{table}

Strong $(3\mod 5)$-arcs in $\PG(2,5)$ with cardinality at most $33$ have been classified without the help of a computer in \cite{landjev2017characterization,rousseva2015structure}. 
In the latter reference nice pictures can be found. In order to provide some more information on the combinatorial structure of the strong $(3\mod 5)$-arcs in $\PG(2,5)$ we 
define the type of a line $L$ as a vector $\left(m_0,\dots, m_5\right)$, where $m_0\ge m_1\ge\dots\ge m_5$ and the $m_i$ are the multiplicities of the points 
on $L$. By assumption we have $m_i\in\{0,1,2,3\}$ for all $0\le i\le 5$. The possible line types are listed in Table~\ref{tab_line_types}.

\begin{table}[htp]
\begin{center}
  \begin{tabular}{lll}
  \hline
  $\cK(L)$ & type of $L$ & name \\ 
  \hline
  3  & $(3,0,0,0,0,0)$ & $A_1$\\ 
     & $(2,1,0,0,0,0)$ & $A_2$\\ 
     & $(1,1,1,0,0,0)$ & $A_3$\\
  \hline
  8  & $(3,3,2,0,0,0)$ & $B_1$\\ 
     & $(3,3,1,1,0,0)$ & $B_2$\\ 
     & $(3,2,2,1,0,0)$ & $B_3$\\ 
     & $(3,2,1,1,1,0)$ & $B_4$\\
     & $(3,1,1,1,1,1)$ & $B_5$\\ 
     & $(2,2,2,2,0,0)$ & $B_6$\\ 
     & $(2,2,2,1,1,0)$ & $B_7$\\ 
     & $(2,2,1,1,1,1)$ & $B_8$\\
  \hline
  13 & $(3,3,3,3,1,0)$ & $C_1$\\ 
     & $(3,3,3,2,2,0)$ & $C_2$\\     
     & $(3,3,3,2,1,1)$ & $C_3$\\
     & $(3,3,2,2,2,1)$ & $C_4$\\
     & $(3,2,2,2,2,2)$ & $C_5$\\
  \hline
  18 & $(3,3,3,3,3,3)$ & $D_1$\\
  \hline
  \end{tabular}     
    \caption{Different line types of strong $(3\mod 5)$-arcs in $\PG(2,5)$.}
  \label{tab_line_types}
\end{center}     
\end{table}

In Tables~\ref{tab_strong_3_mod_5_arcs_pg_2_5_n_18}-\ref{tab_strong_3_mod_5_arcs_pg_2_5_n_93} in the appendix we list the combinatorial details of the strong $(3\mod 5)$-arcs in 
$\PG(2,5)$, i.e., the counts of lines per type and the counts of points per multiplicity, where we give a separate table for each possible cardinality $18\le n\le 93$, $n\equiv 3\pmod 5$. 
As a possible application of this data we mention the following structure result for strong $(3\mod 5)$-arcs in $\PG(3,5)$.
\begin{Lemma}
  \label{lemma_strong_3_mod_5_pg_3_5_lower_bound_card}
  Let $\cK$ be a strong $(3\mod 5)$-arc in $\PG(3,5)$ without a full hyperplane in its support and $H$ be a hyperplane with $\cK(H)\ge 33$. Either we have $\#\cK\ge 125+\cK(H)$ or 
  $\cK|_H$ is one of the following cases:
  \setlength{\tabcolsep}{0.3mm}
  \begin{center}
  \begin{tabular}{|c||c|c|c|c|c|c|c|c|c|c|c|c|c|c|c|c|c||c|c|c|c||c||c|}
  \hline
  $\cK(H)$ & $A_1$ & $A_2$ & $A_3$ & $B_1$ & $B_2$ & $B_3$ & $B_4$ & $B_5$ & $B_6$ & $B_7$ & $B_8$ & $C_1$ & $C_2$ & $C_3$ & $C_4$ & $C_5$ & $D_1$ & $\lambda_0$ & $\lambda_1$ & $\lambda_2$ & $\lambda_3$ & $\#$ & $\#\cK\ge $\\
  \hline
  33 & 0 & 0 & 10 & 0 & 15 & 0  & 0 & 6 & 0 & 0 & 0 & 0 & 0 & 0 & 0 & 0 & 0 & 10 & 15 & 0 & 6 & 1 & 108 \\ 
  33 & 0 & 6 & 4  & 0 & 6  & 12 & 0 & 0 & 0 & 0 & 3 & 0 & 0 & 0 & 0 & 0 & 0 & 12 & 9  & 6 & 4 & 1 & 108 \\ 
  \hline
  43 & 2 & 0 & 0 & 0 & 0  & 25 & 0 & 1  & 0  & 0 & 0 & 0 & 0 & 0 & 0 & 2 & 1 & 10 & 5 & 10 & 6 & 2 & 118 \\ 
  43 & 2 & 0 & 0 & 0 & 25 & 0  & 0 & 2  & 0  & 0 & 0 & 0 & 0 & 0 & 0 & 0 & 2 & 10 & 10 & 0 & 11 & 2 & 118 \\ 
  \hline
  68 & 1 & 0 & 0 & 0 & 0 & 0 & 0 & 1 & 0 & 0 & 0 & 25 & 0 & 0 & 0 & 0 & 4 & 5 & 5 & 0 & 21 & 1 & 168 \\ 
  \hline
  \end{tabular}
  \end{center}
\end{Lemma}
\begin{proof}
  We use Tables~\ref{tab_strong_3_mod_5_arcs_pg_2_5_n_18}-\ref{tab_strong_3_mod_5_arcs_pg_2_5_n_93}, where the possible parameters of the strong $(3\mod 5)$-arcs in $\PG(2,5)$ are listed. 
  Note that we have $\cK(H')\ge 33$ for each hyperplane $H'$ that contains 
  a line of type $B_1$, $B_4$, $B_6$, or $B_7$. So, if $\cK|_H$ contains a line of type $B_1$, $B_4$, $B_6$, or $B_7$, then we have $\#\cK\ge \cK(H)+33\cdot 5-5\cdot 8=125+\cK(H)$. 
  Since each hyperplane $H'$ that contains a line of type $C_2$ or $C_3$ satisfies $\cK(H')\ge 38$, we have $\#\cK\ge \cK(H)+38\cdot 5-5\cdot 13=125+\cK(H)$ if $H$ contains a line of 
  type $C_2$ or $C_3$. If there are no $0$-points in $\cK|_H$, then $\cK$ contains a full hyperplane in its support. All other cases are summarized in the above table. It remains to 
  explain how a lower bound for $\#\cK$ can be obtained. For each line $L$ in $H$ let $m(L)$ denote the minimum cardinality of a strong $(3\mod 5)$-arc in $\PG(2,5)$ that contains 
  a line with the same type as $L$, so that $\#\cK\ge \cK(H)-5\cK(L)+5m(L)$ gives a lower bound. We take the minimum over all possibilities for the type of $L$ in $H$. as an example 
  we consider the two cases where $\cK(H)=33$. There is always a line of type $B_2$ in $H$ with is not contained in an $18$-plane, so that $\cK\ge 33+5\cdot 23-5\cdot 8=108$.       
\end{proof}

\section{Strong $(3\mod 5)$-arcs in $\PG(3,5)$}
\label{section_strong_3_mod_5_in_pg_3_5}
The aim of this section is an exhaustive classification of all strong $(3\mod 5)$-arcs $\cK$ in $\PG(3,5)$. Having Lemma~\ref{lemma_t_mod_q_hyperplane_in_support} at hand, 
or coming from the application of extendable arcs, we will assume that the support of $\cK$ does not contain a full hyperplane. For each strong $(3\mod 5)$-arc $\cK'$ 
in $\PG(2,5)$, see Section~\ref{section_strong_3_mod_5_in_pg_3_5}, of cardinality $m$, the lifting construction in Theorem~\ref{lifted_arcs_construction} gives 
a strong $(3\mod 5)$-arc in $\PG(3,5)$ of cardinality $5m+3$. The question arises if there are any other strong $(3\mod 5)$-arcs in $\PG(3,5)$. It will turn out that there are 
exactly three additional examples, one for the cardinalities $128$, $143$, and $168$, respectively. This contradicts \cite[Theorem 3.3]{landjev2016non}, 
\cite[Theorem 4.1]{landjev2017characterization}, and \cite[Theorem 6]{rousseva2015structure}.  

\begin{Remark}
  Strong $(3\mod 5)$-arcs in $\PG(2,5)$ with cardinality $33$ and type distribution $A_2^6 A_3^4 B_2^6 B_3^{12} B_8^3$ or $A_3^{10} B_2^{15} B_5^6$ of the lines,  
  indeed exist. Generator matrices of the complements of the corresponding minihypers, see Theorem~\ref{thm_pg_2_q}, are given by
  $$
    \begin{pmatrix}
      111111111111111110000\\
      000111222333344441111\\
      134012134023402341234
    \end{pmatrix}  
  $$ 
  and
  $$
    \begin{pmatrix}
      111111111111111110000\\
      000111222333344441111\\
      234134034012301241234\\
    \end{pmatrix}.
  $$
  The existence of these two arcs contradicts \cite[Lemma 4.2]{landjev2017characterization} showing that the corresponding proof is flawed.
\end{Remark}

Given the combinatorial data of the strong $(3 \mod 5)$-arcs in $\PG(2,5)$ in the appendix, it is indeed possible to completely characterize the combinatorial 
data of all $(3 \mod 5)$-arcs in $\PG(3,5)$ with cardinality at most $168$, see \cite[Chapter 8]{kurz2020advanced} for the details. However, the argumentation is rather 
lengthy and error-prone. 
In principle, it should be possible to 
obtain the necessary data from the appendix by computer-free hand-calculations and theoretical arguments. While it is a worthwhile project to obtain a computer-free 
classification of all $(3 \mod 5)$-arcs in $\PG(3,5)$ with cardinality, say, at most $168$, we aim at a full classification using computer enumerations.    

The idea is to automate the kind of reasonings used in e.g.\ \cite{kurz2020advanced,landjev2017characterization,rousseva2015structure}. The main ingredient is the 
known combinatorial data of the strong $(3\mod 5)$-arcs in $\PG(2,5)$, see the appendix. Being general on the one hand and easing the notation on the other hand we 
assume that we are considering strong $(3\mod q)$-arcs in $\PG(3,q)$ for a moment. Given a residual arc $\cK_H$ of a strong $(3 \mod q)$-arc $\cK$ in $\PG(3,q)$, 
i.e.\ a restriction $\cK_H:=\cK|_H$ to a hyperplane $H$, we call the collection of a point $P$ and the $[2]_q=q+1$ lines incident with $P$ a \emph{point-line configuration} 
(also called \emph{line pencil} in the literature). Without restricting to a hyperplane, we call the collection of a point $P$ and the $[3]_q=q^2+q+1$ lines incident with $P$ 
a \emph{full point-line configuration}. For each strong $(3 \mod 5)$-arc in $\PG(2,5)$ we can easily determine the contained point-line configuration. We coarsen our notion 
to a purely combinatorial description, i.e., from now on a (full) point-line configuration is the multiplicity of a point and the counting vector (or distribution) of the 
incident lines per type. As an example we consider the unique strong $(3 \mod 5)$-arc in $\PG(2,5)$ of cardinality $28$. It contains only three different point-line 
configurations:
\begin{itemize}
  \item a $0$-point with line distribution $A_1^2A_3^2B_2^2$;  
  \item a $1$-point with line distribution $A_3^3B_2^3$; and
  \item a $3$-point with line distribution $A_1^1B_2^5$
\end{itemize}   
cf.~\cite{rousseva2015structure}. From now on we assume that we try to classify non-lifted strong $(3\mod 5)$-arcs in $\PG(3,5)$ 
that do not contain a full hyperplane with a fixed target cardinality $\#\cK$, i.e., we prescribe the cardinality. The assumption that $\cK$ does not contain 
a full hyperplane excludes a few of the strong $(3\mod 5)$-arcs in $\PG(2,5)$, e.g.\ the residual arc of cardinality $93$ cannot occur, so that $\cK(H)\le 68$.

Our first kind of reasonings are simple cardinality bounds. Consider a line $L$ of a given type and the hyperplanes $H_0,\dots,H_q$ through it. Since 
\begin{equation}
  \#\cK =\sum_{i=0}^q \cK\!\left(H_i\right) \,-\,q\cK(L),
\end{equation}
it is sufficient to know the cardinality $\cK(L)$ of the line and the possible cardinalities $\cK\!\left(H_i\right)$ of the hyperplanes through $L$ in order 
to exclude a few cardinalities. As an example we consider a line of type $C_4$ or $C_5$, i.e.\ a special line of cardinality $13$. The possible cardinalities of the hyperplanes 
containing a line of type $C_4$ are contained in $\{33, 38,\dots, 63\}$ and those containing a line of type $C_5$ are contained in $\{38, 38,\dots, 63\}$. Thus, if $\#\cK<133$ 
or $\#\cK>313$, then $\cK$ cannot contain a line of type $C_4$ and if $\#\cK<163$ or $\#\cK>313$, then $\cK$ cannot a line of type $C_5$. So, prescribing the target cardinality 
$\#\cK$ of the arc may result in some excluded line types. As an example we consider $\#\cK=128$ and state that the line types contained in $\left\{B_1,B_4,B_6,B_7,C_1,C_2,C_3,C_4\right\}$ 
are excluded with the above reasoning. Clearly, a point-line configuration that contains an excluded line is excluded itself and a residual arc that contains an excluded 
point-line configuration is excluded itself. For our example $\#\cK=128$ the $8$ excluded line types imply the exclusion of $1221$ out of the $1288$ possible point-line configurations, 
which then imply the exclusion of $166$ out of the $178$ residual arcs containing at least one point with multiplicity $0$. Of course the exclusion of point-line configurations 
and residual arcs may imply tighter restrictions for the sets of possible cardinalities of a hyperplane containing a line of a certain type, so that the above reasoning 
may be applied iteratively. In our example the line types $C_5$ and $D_1$ are excluded in the next iteration, which then implies the exclusion of $19$ further point-line 
configurations and $6$ residual arcs. Instead of starting from a line, we can also start from a point-line configuration of a certain type, fix a contained line type 
and consider the possible cardinalities of the hyperplanes through this line. The difference to the previous reasoning is that we have fixed one of the hyperplane cardinalities, 
e.g.\ $\cK(H_0)$. If the target cardinality $\#\cK$ cannot be reached, then the starting point-line configuration can be excluded. In our example $13$ further point-line 
configurations are excluded this way. We can also exclude point-line configurations which are only contained in already excluded residual arcs ($13$ cases in our example). 
We remark that the execution of the checks described above is computationally cheap, i.e., the running time is negligible. So, we apply them recursively until no more 
new exclusions are found. In our example $\#\cK=128$ we end up with $7$ remaining line types, $22$ remaining point-line configurations, and $6$ remaining residual arcs.       

Next we try to enumerate candidates for full point-line configurations. Starting from a line type and a point multiplicity contained in this line type, we can loop over 
all multisets of $q+1$ point-line configurations that contain the prescribed line type (and point multiplicity) and go in line with $\#\cK$. Having a candidate for a 
full point-line configuration at hand we can eventually exclude it if it corresponds to a lifted arc. This local information must of course be consistent. So, assume that 
$f$ is a full point-line configuration for a line of type $t$ centered at a point of multiplicity $m$. If $t'\neq t$ is another line type containing a point of multiplicity $m$, 
then $f$ has to be contained in the list of possible full point-line configurations for a line of type $t'$ centered at a point of multiplicity $m$. If this is not the 
case, then we can remove the full point-line configuration $f$ from the list for $(m,t)$. In our example with $\cK=128$ we first enumerate $12$ possible 
full point-line configurations for a line of type $A_1$ centered at a line of multiplicity $0$. The consistency check leaves only the possibility
$A_1^6 A_2^{12} A_3^6 B_2^3 B_3^4$ with $\left(\lambda_0,\lambda_1,\lambda_2,\lambda_3\right)=(80,40,20,16)$. Lines of type $B_5$ centered at a point of multiplicity $1$ 
do not admit a consistent full point-line configuration (in the above sense). Thus lines of type $B_5$ are excluded and we can restart with the simple cardinality bounds 
ending up with $6$ remaining line types, $16$ remaining point-line configurations, and $4$ remaining residual arcs. These possibilities correspond to exactly 
those substructures that occur in the strong $(3 \mod 5)$-arc of cardinality $128$ in $\PG(3,5)$ that is combinatorially described in 
Theorem~\ref{thm_3_mod_5_arcs_pg_3_5_card_small_or_large}. 

It may happen that the exhaustive enumeration of candidates for full point-line configurations via line types results in no further exclusions or is computationally too expensive. 
Alternatively, we may also start from a residual arc. For each contained point-line configuration we can exhaustively enumerate the candidates for full point-line configurations 
via the contained line types. The computational advantage for the enumeration is that one point-line configuration is already fixed. For the same point-line 
configuration $t$ the possibilities for the full point-line configurations have to be consistent for all lines contained in $t$. As a further consistency check we also compute the 
point multiplicity distributions $\mathbf{\lambda}=\left(\lambda_0,\lambda_1,\lambda_2,\lambda_3\right)$ for the full point-line configurations. A given full point-line 
configuration $f$ with point multiplicity distribution $\mathbf{\lambda}$ is inconsistent if there exists a point-line configuration $t$ (in our currently considered residual arc) 
and a line type in $t$ that does not admit a full point-line configuration $f'$ with point multiplicity distribution $\mathbf{\lambda}$. As an example we consider 
$\#\cK=143$. Here the checks based on simple cardinality bounds leave $5$ lines, $15$ point-line configurations, and $6$ residual arcs. The exhaustive enumeration of candidates 
for full point-line configurations via line types results leaves at least one possible full point-line configuration for each of the remaining line types.\footnote{More  
sophisticated conclusions might still be drawn. For lines of type $D_1$ (centered at a point of multiplicity $3$) there remains a unique possibility leading to 
$\mathbf{\lambda}=(75,50,0,31)$. For lines of type $B_5$  (centered at a point of multiplicity $3$) there remains a unique possibility leading to $\mathbf{\lambda}=(65,65,0,26)$. 
In all other cases we have more than one remaining possibility for the full point-line configurations. Nevertheless, we might conclude that $\cK$ cannot contain a line of type 
$B_5$ and a line of type $D_1$.} The refined exhaustive enumeration of candidates for full point-line configurations via residual arcs can exclude the two remaining residual 
arcs of cardinality $43$. Applying the checks based on simple cardinality bounds leave $4$ line types, $9$ point-line configurations, and $3$ residual arcs. These possibilities 
correspond to exactly those substructures that occur in the strong $(3 \mod 5)$-arc of cardinality $143$ in $\PG(3,5)$ that is combinatorially described in 
Theorem~\ref{thm_3_mod_5_arcs_pg_3_5_card_small_or_large}. If some point-line configurations are already excluded, then the following check can sometimes eliminate a few possible 
full point-line configurations. If $l$ and $l'$ are two lines with different types and all point-line configurations that can be spanned by $l$ and $l'$ are already 
excluded, then in a given full point-line configuration not both entries, for the type of $l$ and for the type of $l'$, can be strictly positive. If $l$ and $l'$ are of 
the same type, then the condition for a full point-line configuration is that the corresponding entry is at most $1$. (Of course, the multiplicity of the intersection 
point of $l$ and $l'$ must coincide with the multiplicity of the center of the full point-line configuration.)        

For cardinalities $\#\cK$ that are neither small nor large the computation times for the exhaustive enumeration of full point-line configurations dramatically 
increase, so that further tools are needed. One, conceptionally easy, approach is to prescribe a point multiplicity distribution $\mathbf{\lambda}$ in a separate 
computation. If the computation ends up with a computational impossibility proof, then we can exclude this specific point multiplicity distribution $\mathbf{\lambda}$, 
which results in excluded full point-line configurations in the subsequent computations. As an example we consider cardinality $\#\cK=173$. Applying the 
exhaustive enumeration of candidates for full point-line configurations via line types results it turns out that for line type $C_5$ centered at points 
of multiplicity $3$ there are just two different full point-line configurations: $(6,0,0,1,1,0,0,10,5,0,0,0,0,0,0,4,4)$ and $(5,0,0,1,1,0,0,10,7,0,0,0,0,0,0,2,5)$. 
Both possibilities correspond to point multiplicity distribution $\mathbf{\lambda}=(64,33,37,22)$. This specific vector for $\mathbf{\lambda}$ can be computationally 
excluded in just a few seconds. Then explicitly excluding point multiplicity distribution $\mathbf{\lambda}=(64,33,37,22)$ computationally implies the existence of 
lines of type $C_5$, so that we can continue from an easier starting position.

\begin{Theorem}
  \label{thm_3_mod_5_arcs_pg_3_5_card_small_or_large}
  Let $\cK$ be a strong $(3\mod 5)$-arc in $\PG(3,5)$ that is neither lifted nor contains a full hyperplane. Then either 
  $178\le \#\cK\le 273$ or $\cK$ is given by one of the following three possibilities: 
  \begin{itemize}
    \item $\#\cK=128$, $\left(a_{18},a_{23},a_{28},a_{33}\right)=(20,80,16,40)$;
    \begin{center}\hspace*{-5mm}
    \begin{tabular}{lcccc}
    \hline 
    $\cK(P)$ for $P\!\in\!\cP$ & 0 & 1 & 2 & 3\\ 
    \hline
    $\#$              & 80 & 40 & 20 & 16 \\ 
    line distr. & $A_1^6 A_2^{12} A_3^6 B_2^3 B_3^4$ & $A_2^6 A_3^{12} B_2^6 B_3^4 B_8^3$ & $A_2^{12}B_3^{16}B_8^3$ & $A_1^6 B_2^{15} B_3^{10}$\\
    \hline
    \end{tabular}
    \end{center}
    \begin{center}\hspace*{-5mm}
    \begin{tabular}{lcccccc}
    \hline 
    line type & $A_1$ & $A_2$ & $A_3$ & $B_2$ & $B_3$ & $B_8$\\ 
    \hline
    $\#$              & 96 & 240 & 160 & 120 & 160 & 30 \\      
    hyp.~distr. & $23^5 28^1$ & $18^1 23^4 33^1$ & $18^2 23^2 28^1 33^1$ & $23^2 28^2 33^2$ & $23^3 33^3$ & $18^2 33^4$\\ 
    \hline
    \end{tabular}
    \end{center}
      Generator matrix given by the concatenation of
      $${\tiny
        \left(\begin{smallmatrix}
        0000000000000000000000000000000001111111111111111111111111111111\\  
        0001111111111111111111111111111110000000000000000000011111111111\\  
        1110001111111122222222333333334440001111111122233344400011122233\\  
        1140130112233311233344011122232240221112344400113311123333302300\\  
        \end{smallmatrix}\right)}
      $$
      and
      $${\tiny
        \left(\begin{smallmatrix}
        1111111111111111111111111111111111111111111111111111111111111111\\
        1111111112222222222222223333333333333333333344444444444444444444\\
        3333334440001112223334440001112223334444444400011111111222333444\\
        1112334441240120132440330244440443330111224412201113334002111144\\
        \end{smallmatrix}\right)}
      $$  
      with a corresponding automorphism group of order $7680$.
    \item $\#\cK=143$, $\left(a_{18},a_{23},a_{28},a_{33}\right)=(26,0,65,65)$;
    \begin{center}\hspace*{-5mm}
    \begin{tabular}{lcccc}
    \hline 
    $\cK(P)$ for $P\!\in\!\cP$ & 0 & 1 & 2 & 3\\ 
    \hline
    $\#$ & 65 & 65 & 0 & 26 \\ 
    line distr. & $A_1^6 A_3^{15} B_2^{10}$ & $A_3^{15} B_2^{10} B_5^6$ & & $A_1^3 B_2^{25} B_5^{3}$\\ 
    \hline
    \end{tabular}
    \end{center}
    \begin{center}\hspace*{-5mm}
    \begin{tabular}{lcccc}
    \hline 
    line type & $A_1$ & $A_3$ & $B_2$ & $B_5$ \\ 
    \hline 
    $\#$              & 78 & 325 & 325 & 78 \\      
    hyp.~distr. & $18^1 28^5$ & $18^2 28^2 33^2$ & $28^3 33^3$ & $18^1 33^5$\\
    \hline
    \end{tabular}
    \end{center}
      Generator matrix given by the concatenation of
      $${\tiny
        \left(\begin{smallmatrix}
        000000000000000000000000000000000111111111111111111111111111111111111111\\ 
        000000001111111111111111111111111000000000000000000001111111111111111111\\ 
        011111110000000112222222334444444000000011223344444440000000112233333334\\ 
        100033340122234020111444030123334000122223030111123330123334043400012341\\ 
        \end{smallmatrix}\right)}
      $$
      and
      $${\tiny
        \left(\begin{smallmatrix}
        11111111111111111111111111111111111111111111111111111111111111111111111\\
        11111122222222222222222222333333333333333333333333344444444444444444444\\
        44444400000001122222223344001122222223333333444444400112222222333333344\\
        11222402223331200011131302043400012341112224012333403011112333000122223\\
        \end{smallmatrix}\right)}
      $$  
      with a corresponding automorphism group of order $62400$.
    \item $\#\cK=168$, $\left(a_{28},a_{33},a_{43}\right) =(60,60,36)$;
    \begin{center}\hspace*{-5mm}
    \begin{tabular}{lcccc}
    \hline 
    $\cK(P)$ for $P\!\in\!\cP$ & 0 & 1 & 2 & 3\\ 
    \hline
    $\#$ & 60 & 60 & 0 & 36 \\ 
    line distr. & $A_1^6 A_3^{10} B_2^{15}$ & $A_3^{10}B_2^{15}B_5^6$ & & $A_1^2B_2^{25} B_5^2 D_1^2$ \\ 
    \hline
    \end{tabular}
    \end{center}
    \begin{center}\hspace*{-5mm}
    \begin{tabular}{lccccc}
    \hline 
    line type & $A_1$ & $A_3$ & $B_2$ & $B_5$ & $D_1$\\ 
    \hline 
    $\#$              & 72 & 200 & 450 & 72 & 12 \\      
    hyp.~distr. & $28^5 43^1$ & $28^3 33^3$ & $28^2 33^2 43^2$ & $33^5 43^1$ & $43^6$\\ 
    \hline
    \end{tabular}
    \end{center}
      Generator matrix given by the concatenation of
      $${\tiny
        \left(\begin{smallmatrix}
        00000000000000000000000000000000000000000001111111111111111111111111111111111111111111\\ 
        00000000000000000011111111111111111111111110000000000000000000000000111111111111111111\\ 
        00011111111111111100000111112222233333444440000011111222223333344444000001111122222333\\ 
        11100011122233344403444034440344403444034440222412444111340133300023000140344423334122\\ 
        \end{smallmatrix}\right)}
      $$
      and
      $${\tiny
        \left(\begin{smallmatrix}
        1111111111111111111111111111111111111111111111111111111111111111111111111111111111\\
        1111111222222222222222222222222233333333333333333333333334444444444444444444444444\\
        3344444222223333333333333334444400000111112222233333444440000011111222223333344444\\
        2301112012340001112223334440123412223233340344400014011120002301333111341244402224\\
        \end{smallmatrix}\right)}
      $$
      with a corresponding automorphism group of order $57600$.
  \end{itemize}
\end{Theorem}
\begin{proof}
  We have implemented the exclusion arguments described above in a computer program. Running the program for all cardinalities with 
  either $\#\cK\le 173$ or $\#\cK\ge 278$ resulted in a computational proof of impossibility in all cases except $\#\cK\in\{128,143,168,173\}$. 
  For $\#\cK=128$ the output of the program is:
{\tiny\begin{verbatim}
6 line types remain.
16 point-line types remain.
4 residual arcs remain.
Remaining line 0 with cardinality 3: 5 0 0 1 
Remaining line 2 with cardinality 3: 4 1 1 0 
Remaining line 5 with cardinality 3: 3 3 0 0 
Remaining line 6 with cardinality 8: 0 4 2 0 
Remaining line 7 with cardinality 8: 2 1 2 1 
Remaining line 8 with cardinality 8: 2 2 0 2 
Remaining point-line configuration 9: 2 2 2 2 2 6 6 
Remaining point-line configuration 10: 1 2 5 5 5 5 6 
Remaining point-line configuration 11: 0 2 2 2 5 5 5 
Remaining point-line configuration 12: 3 0 0 7 7 8 8 
Remaining point-line configuration 13: 2 2 2 2 7 7 7 
Remaining point-line configuration 14: 0 0 0 2 2 5 7 
Remaining point-line configuration 15: 1 2 2 5 5 7 8 
Remaining point-line configuration 16: 0 0 2 2 2 2 8 
Remaining point-line configuration 17: 3 0 8 8 8 8 8 
Remaining point-line configuration 18: 0 0 0 5 5 8 8 
Remaining point-line configuration 19: 1 5 5 5 8 8 8 
Remaining point-line configuration 21: 2 2 6 7 7 7 7 
Remaining point-line configuration 25: 0 2 2 5 7 7 8 
Remaining point-line configuration 35: 1 5 5 6 7 7 8 
Remaining point-line configuration 38: 3 7 7 7 8 8 8 
Remaining point-line configuration 55: 1 2 2 6 6 8 8 
Remaining hyperplane 3 with cardinality 18: 9 10 11 
Remaining hyperplane 4 with cardinality 23: 12 13 14 15 16 
Remaining hyperplane 5 with cardinality 28: 17 18 19 
Remaining hyperplane 9 with cardinality 33: 21 25 38 35 55 
\end{verbatim}}
  For $\#\cK=143$ the output of the program is:
{\tiny\begin{verbatim}
4 line types remain.
9 point-line types remain.
3 residual arcs remain.
Remaining line 0 with cardinality 3: 5 0 0 1 
Remaining line 3 with cardinality 8: 0 5 0 1 
Remaining line 5 with cardinality 3: 3 3 0 0 
Remaining line 8 with cardinality 8: 2 2 0 2 
Remaining point-line configuration 6: 3 0 0 0 3 3 3 
Remaining point-line configuration 7: 1 3 5 5 5 5 5 
Remaining point-line configuration 8: 0 0 5 5 5 5 5 
Remaining point-line configuration 17: 3 0 8 8 8 8 8 
Remaining point-line configuration 18: 0 0 0 5 5 8 8 
Remaining point-line configuration 19: 1 5 5 5 8 8 8 
Remaining point-line configuration 56: 1 3 3 5 5 8 8 
Remaining point-line configuration 57: 3 3 8 8 8 8 8 
Remaining point-line configuration 58: 0 5 5 5 8 8 8 
Remaining hyperplane 2 with cardinality 18: 6 7 8 
Remaining hyperplane 5 with cardinality 28: 17 18 19 
Remaining hyperplane 10 with cardinality 33: 56 57 58 
\end{verbatim}}
  For $\#\cK=168$ the output of the program is:
{\tiny\begin{verbatim}
5 line types remain.
10 point-line types remain.
4 residual arcs remain.
Remaining line 0 with cardinality 3: 5 0 0 1 
Remaining line 1 with cardinality 18: 0 0 0 6 
Remaining line 3 with cardinality 8: 0 5 0 1 
Remaining line 5 with cardinality 3: 3 3 0 0 
Remaining line 8 with cardinality 8: 2 2 0 2 
Remaining point-line configuration 17: 3 0 8 8 8 8 8 
Remaining point-line configuration 18: 0 0 0 5 5 8 8 
Remaining point-line configuration 19: 1 5 5 5 8 8 8 
Remaining point-line configuration 56: 1 3 3 5 5 8 8 
Remaining point-line configuration 57: 3 3 8 8 8 8 8 
Remaining point-line configuration 58: 0 5 5 5 8 8 8 
Remaining hyperplane 5 with cardinality 28: 17 18 19 
Remaining hyperplane 10 with cardinality 33: 56 57 58 
Remaining hyperplane 61 with cardinality 43: 487 488 489 490 
Remaining hyperplane 62 with cardinality 43: 487 488 489 490
\end{verbatim}}
In order to exhaustively enumerate the strong $(3\mod 5)$-arcs in $\PG(3,5)$ with a cardinality in $\{128,143,168\}$ we utilize an integer linear programming (ILP) formulation 
based on binary indicator variables $x_{P,i}\in\{0,1\}$ for the combinations of the $156$ points $P\in\cP$ and the point multiplicities $i\in \{0,1,2,3\}$ satisfying 
\begin{equation}
  \sum_{i=0}^3 x_{P,i}=1
\end{equation}
for each point $P\in\cP$. For each line $L$ we have 
\begin{equation}
  \sum_{P\in L}\sum_{i=0}^3 x_{P,i}=3+5y_L,
\end{equation} 
where $y_L\in\N_0$ with $y_L\le 3$. Similarly, for each hyperplane $H$ we have
\begin{equation}
  \sum_{P\in H}\sum_{i=0}^3 x_{P,i}=18+5z_H,
\end{equation} 
where $z_H\in\N_0$ with $z_H\le 15$. The restrictions on the cardinality of lines and hyperplanes implied by the above data result in improved upper bounds 
for $y_L$ and $z_H$, respectively. I.e., if the maximum line multiplicity is $8$, then we can use $y_L\le 1$ for all lines $L$, and if the maximum hyperplane 
multiplicity is $43$, then we can use $z_H\le 5$ for all hyperplanes $H$. Additionally we implement the stored isomorphism type of the residual arc with the largest 
possible cardinality by prescribing the $x_{P,i}$ for all points $P$ contained in an arbitrary but fixed hyperplane $H^\star$. Using \texttt{CPLEX} 
and the \texttt{populate} command we enumerate all feasible solutions and filter the corresponding codes up to isomorphy. For $\#\cK=128$ and $\#\cK=143$ there is a unique 
possible residual arc of cardinality $33$ and all ILP solutions correspond to the stated example. For $\#\cK=168$ there are two possible residual arcs of cardinality 
$43$. For hyperplane nr~$61$ there are no ILP solutions at all. 
For hyperplane nr~$62$ (note the equal combinatorial data) there are ILP solutions, 
which correspond to 
arcs with a generator matrix given by the concatenation of
$$
  {\tiny\left(\begin{smallmatrix} 
    000000000000000000000000000000000000000000011111111111111111111111111111111111111111\\
    000000000000000000111111111111111111111111100000000000000000000000001111111111111111\\
    000111111111111111000001111122222333334444400000111112222233333444440000011111222223\\
    111000111222333444034440344403444034440344401234012340123401234012340111201112011120\\
  \end{smallmatrix}\right)}
$$
and
$$
  {\tiny\left(\begin{smallmatrix}
    111111111111111111111111111111111111111111111111111111111111111111111111111111111111\\
    111111111222222222222222222222222233333333333333333333333334444444444444444444444444\\
    333344444000001111122222333334444400000111112222233333444440000011111222223333344444\\
    111201112022240222402224022240222400023000230002300023000230111201112011120111201112\\
  \end{smallmatrix}\right)}   
$$
or
$$
  {\tiny\left(\begin{smallmatrix} 
    000000000000000000000000000000000000000000011111111111111111111111111111111111111111\\
    000000000000000000111111111111111111111111100000000000000000000000001111111111111111\\
    000111111111111111000001111122222333334444400000111112222233333444440000011111222223\\
    111000111222333444034440344403444034440344402224124441113401333000230001403444233341\\
  \end{smallmatrix}\right)}
$$   
and
$$
  {\tiny\left(\begin{smallmatrix}
    111111111111111111111111111111111111111111111111111111111111111111111111111111111111\\
    111111111222222222222222222222222233333333333333333333333334444444444444444444444444\\
    333344444222223333333333333334444400000111112222233333444440000011111222223333344444\\
    222301112012340001112223334440123412223233340344400014011120002301333111341244402224\\
  \end{smallmatrix}\right)}.
$$
The first arc is lifted\footnote{In the special situation we could have avoided the corresponding ILP solutions by requiring a lower cardinality bound of $28$ for every  
hyperplane. For the general case we remark that it is easily possible to exclude lifted arcs by additional constraints (and variables). However, such a model extension 
resulted in significantly increased computation times in our experiments, so that we do not go into details here.} and the second is the stated example.

For $\#\cK=173$ we exclude the specific cases of a point multiplicity distribution
\begin{eqnarray*}
  \mathbf{\lambda}&\in&\big\{(64,33,37,22), (55,50,30,21), (54,53,27,22),\\ 
  && (60,40,35,21), (58,46,29,23), (63,36,34,23)\big\}
\end{eqnarray*}
in separate computations. Afterwards the program yielded a computational proof of non-existence.  
\end{proof}
We remark that the computations times for $\#\cK\le 158$ (without the ILP computations) are negligible. For $\#\cK=163$ the computation took 2~minutes,  
for $\#\cK=168$ the computation took less than 27~minutes, and for $\#\cK=173$ the computation took slightly less than 869~minutes. The latter computation 
times heavily depend on the algorithmic details of the implementation as well as on the order which specific substructure is tried to exclude first and 
parameters controlling the abort of exhaustive enumerations. We remark that the exclusion of specific point multiplicity distributions can of course be 
performed in parallel. Using the same techniques as in the proof of Theorem~\ref{thm_3_mod_5_arcs_pg_3_5_card_small_or_large}, but without explicitly 
listing the separately excluded point multiplicity distributions $\mathcal{E}_n$, we computationally show that $178\le\#\cK\le 273$ is impossible. 
In Table~\ref{tab_cluster} we state the number of cases $\#\mathcal{E}_n$ and the corresponding computation times. All jobs ran in parallel on a 
large-scale computing cluster of the University of Bayreuth. Candidates for $\#\mathcal{E}_n$ were obtained by sampling and choosing the most frequent ones.  

\begin{table}[htp]
\begin{center}
  \begin{tabular}{lrrrrrrrrrr}
  \hline
  $n$               & 178  & 183 & 188 & 193 & 198  & 203  & 208  & 213  & 218  & 223  \\ 
  $\#\mathcal{E}_n$ & 31   & 36  & 46  & 75  & 180  & 174  & 176  & 179  & 177  & 179  \\ 
  time in h         & 3078 & 351 & 998 & 972 & 1434 & 1787 & 2368 & 2661 & 3214 & 3110 \\
  \hline
  $n$               & 228  & 233  & 238  & 243  & 248  & 253  & 258  & 263  & 268  & 273  \\ 
  $\#\mathcal{E}_n$ & 176  & 180  & 177  & 170  & 176  & 170  & 161  & 173  & 148  & 111  \\ 
  time in h         & 3477 & 3448 & 3396 & 3150 & 2848 & 2042 & 1752 & 855  & 911  & 683 \\
  \hline
  \end{tabular}
  \caption{Details for the computations for $178\le \#\cK\le 273$.}
  \label{tab_cluster}
\end{center}
\end{table}

\begin{Theorem}
  Let $\cK$ be a strong $(3\mod 5)$-arc in $\PG(3,5)$ that is neither lifted nor contains a full hyperplane. Then $\#\cK\in\{128,143,168\}$ and $\cK$ 
  is given as specified in Theorem~\ref{thm_3_mod_5_arcs_pg_3_5_card_small_or_large}. 
\end{Theorem}

While the utilized method is applicable in principle also for larger field sizes, it remains an algorithmical challenge to speed up the computations so that 
the strong $(3\mod 7)$-arc in $\PG(3,7)$ may be classified in reasonable time. As the found non-lifted strong $(3\mod 5)$-arcs in $\PG(3,5)$ have quite some 
automorphisms one may also heuristically search non-lifted $(3\mod q)$-arcs in $\PG(3,q)$ by prescribing suitable subgroups of the automorphism group or try 
to unveil their geometric structure

\begin{Conjecture}
  Every strong $(3\mod 5)$-arc in $\PG(v-1,5)$ is lifted for $v\ge 5$.
\end{Conjecture}

\section{The non-existence of $\mathbf{(104,22)}$-arcs in $\mathbf{\PG(3,5)}$}
\label{sec_non-existence_104_22_4_5}

The aim of this section is to show the non-existence of a $(104,22)$-arc in $\PG(3,5)$, i.e., to fix the gap in the corresponding proof of \cite{landjev2016non} due to the 
flawed classification of strong $(3\mod 5)$ in $\PG(3,5)$. By $\gamma_i$ we denote the maximum multiplicity of an $i$-space, i.e., $\gamma_1$ is the maximum point multiplicity.  

\begin{Lemma}
  \label{lemma_residual}
  The maximum multiplicity of a hyperline, i.e., a subspace of codimension $2$, in an $m$-hyperplane of an $(n,\le s)$-arc $\cK$ in $\PG(v-1,q)$, where $v\ge 3$, is 
  at most $\left\lfloor (sq+m-n)/q\right\rfloor$.
\end{Lemma}
\begin{proof}
  Let $S$ be an arbitrary hyperline and $H_0,\dots, H_q$ the $q+1$ hyperplanes through $S$. With this and $\cK(H_0)=m$ we have 
  $$
    n=\sum_{i=0}^q \cK(H_i)-q\cdot \cK(S)\le m+q\cdot s-q\cdot \cK(S),  
  $$
  so that
  $$
    \cK(S)\le \frac{qs+m-n}{q}.
  $$
  Note that $\cK(S)$ is a non-negative integer and $m\le s$.
\end{proof}

\begin{Lemma}
  \label{lemma_104_22_3_5_straightforward}
  Let $\cK$ be a $(104,22)$-arc in $\PG(3,5)$ with spectrum $\left(a_i\right)$. Then:
  \begin{enumerate}
    \item[(a)] The maximal multiplicity of a line in an $m$-plane is $\left\lfloor(6+m)/5\right\rfloor$.
    \item[(b)] $\gamma_1=1$, $\gamma_2=5$, $\gamma_3=22$.
    \item[(c)] There do not exist planes with $2$, $3$, $7$, $8$, $12$, $13$, $17$, or $18$ points.
  \end{enumerate}  
\end{Lemma}
\begin{proof}$\,$\\[-3mm]
  \begin{enumerate}
    \item[(a)] Apply Lemma~\ref{lemma_residual}. 
    \item[(b)] $\gamma_3=22$ follows from the definition of the arc. (a) implies $\gamma_2\le 5$. If $\gamma_2\le 4$, then considering the $31$ lines 
               through a point of multiplicity at least $1$ would yield $\#\cK\le 1+31\cdot 3<104$. Obviously $\gamma_1\ge 1$. Considering the $31$ lines 
               through a point of multiplicity at least $2$ would yield $\#\cK\le 5+31\cdot 3<104$.
    \item[(c)] Using (a), this follows from the non-existence of $(2,1)$-, $(7,2)$-, $(12,3)-$, and $(17,4)$-arcs in $\PG(2,5)$.
  \end{enumerate}
\end{proof}    
Thus, a $(104,22)$-arc $\cK$ in $\PG(3,5)$ is $3$-quasi-divisible with divisor $5$ and gives rise to a $(3\mod 5)$-arc $\widetilde{\cK}$ in $\PG(3,5)$. 
Since the possibilities for $\widetilde{\cK}$ are completely classified, we can utilize an integer linear programming (ILP) formulation for $\cK$ given $\widetilde{\cK}$. 
We use binary variables $x_P\in\{0,1\}$, with the meaning $x_P=\cK(P)$, for all $P\in\cP$. For each hyperplane $H\in\cH$ we require
\begin{equation}
  5y_H+\sum_{P\in\cP\,:\, P\le H} x_P=22-\widetilde{\cK}(\widetilde{H}),
\end{equation} 
where $y_H=0$ if $\widetilde{\cK}(\widetilde{H})=0$. In general the $y_H$ are non-negative integers and model the fact that $\cK(H)\equiv 22-\widetilde{\cK}(\widetilde{H})\pmod 5$. 
If $\cK(H)\equiv 22\pmod 5$, then $\cK(H)=22$ due to Lemma~\ref{lemma_104_22_3_5_straightforward}, which is translated to $y_H=0$ in that case. 

The infeasibility of those ILPs for all different choices for $\widetilde{\cK}$ yields:
\begin{Theorem}
  \label{thm_no_104_22}
  No $(104,22)$-arc in $\PG(3,5)$ exists.
\end{Theorem} 
Using \texttt{CPLEX} all ILPs were solved in less than 2~minutes in total, which is mainly due to the fact that many hyperplanes have to be of multiplicity $22$ and these are exactly 
characterized by $\widetilde{\cK}$.

Similarly, as Lemma~\ref{lemma_104_22_3_5_straightforward} we can show:
\begin{Lemma}
  \label{lemma_103_22_3_5_straightforward}
  Let $\cK$ be a $(103,22)$-arc in $\PG(3,5)$ with spectrum $\left(a_i\right)$. Then:
  \begin{enumerate}
    \item[(a)] The maximal multiplicity of a line in an $m$-plane is $\left\lfloor(7+m)/5\right\rfloor$.
    \item[(b)] $\gamma_1=1$, $\gamma_2=5$, $\gamma_3=22$.
    \item[(c)] There do not exist planes with $2$, $7$, $12$, or $17$.
  \end{enumerate}  
\end{Lemma}
So, if $\cK$ is a $(103,22)$-arc $\cK$ in $\PG(3,5)$, we cannot assume directly that $\cK$ is $3$-quasi-divisible with divisor $5$. However, under the additional assumption 
$a_3=a_8=a_{13}=a_{18}=0$ it is and we can again apply ILP computations to obtain:
\begin{Proposition}
  If $\cK$ is a $(103,22)$-arc $\cK$ in $\PG(3,5)$, then there exists a hyperplane $H$ with $\cK(H)\in\{3,8,13,18\}$.
\end{Proposition}  

\subsection{Theoretical shortcuts}
\label{subsec_shortcuts}
The original proof of Theorem~\ref{thm_no_104_22} in \cite{landjev2016non} was completely free of computer calculations. Since we have used massive computer calculations 
in the classification of the $(3 \mod 5)$-arcs in $\PG(3,5)$ we cannot reach this worthwhile goal in this article. However, starting from a partial classification of 
all strong $(3\mod 5)$-arcs in $\PG(3,5)$ with cardinality less than $163$, the ILP computations can be restricted to the two non-lifted $(3\mod 5)$-arcs of cardinalities 
$128$ and $143$ in $\PG(3,5)$. As some details in \cite{landjev2016non} are left to the reader and a very few minor typos and computational errors may deter the hurried reader 
from seeing all details, we give a full proof along the ideas presented in \cite{landjev2016non}. However, we slightly reduce the used techniques. 

As mentioned, we currently still need the following conclusion from ILP computations.
\begin{Lemma}
  \label{lemma_tilde_k_neq_128_143}
  Let $\cK$ be a $(104,22)$-arc in $\PG(3,5)$ and $\widetilde{\cK}$ be the corresponding dual strong $(3\mod 5)$-arc. Then $\widetilde{\cK}$ is either 
  lifted or $\widetilde{\cK} \notin\{128,143\}$.
\end{Lemma}

\begin{Lemma}(\cite[Lemma 4.2]{landjev2016non})
  \label{lemma_no_threefold_line_hyperplane}
  Let $\cK$ be a $(104,22)$-arc in $\PG(3,5)$ and $\widetilde{\cK}$ be the corresponding dual strong $(3\mod 5)$-arc. Then, there exists no plane $\widetilde{\pi}$ in the 
  dual space such that $\widetilde{\cK}|_{\widetilde{\pi}}$ is $3\chi_{\widetilde{L}}$ for some line $\widetilde{L}$ in the dual space.
\end{Lemma}
\begin{proof}
  Let $P$ be the point corresponding to $\widetilde{\pi}$ and $L$ be the line corresponding to $\widetilde{L}$. Summing up the multiplicities of all all planes through $P$ gives
  $$
    \sum_{H\in\cH\,:\,P\le H} \cK(H)=6\#\cK+25\cK(P)
  $$
  and summing up the multiplicities of all all planes through $L$ gives
  $$
    \sum_{H\in\cH\,:\,L\le H} \cK(H)=\#\cK+5\cK(L).
  $$
  Since $\widetilde{L}$ is incident with $\widetilde{\pi}$, $P$ is incident with $L$. Those hyperplanes $H$ through $P$ that do not contain $L$, correspond to 
  points $\widetilde{H}$ in the dual space that are not contained on $\widetilde{L}$, so that $\widetilde{\cK}(\widetilde{H})=0$ and $H$ is a maximal plane, 
  i.e., $\cK(H)=22$. Thus, all $[3]_5-[2]_5=25$ hyperplanes through $P$ that do not contain $L$ are $22$-planes and we have
  $$
     6\#\cK+25\cK(P) =25\cdot 22+\#\cK+5\cK(L),
  $$  
  which is equivalent to
  $$
    25\cK(P)=30+5\cK(L).
  $$
  Since $\cK(P)\in\{0,1\}$ and $\cK(L)\ge 0$, this is a contradiction. 
\end{proof}

\begin{Lemma}
  \label{lemma_lower_bound_tilde_k}
  Let $\cK$ be a $(104,22)$-arc in $\PG(3,5)$ and $\widetilde{\cK}$ be the corresponding dual strong $(3\mod 5)$-arc, then $\#\widetilde{\cK}\ge 163$.
\end{Lemma}
\begin{proof}
  Due to the non-existence of a $(105,22)$-arc in $\PG(3,5)$ we can assume that $\widetilde{\cK}$ does not cannot contain a full hyperplane in its support, see 
  Theorem~\ref{thm_quasidivisible_extendability}. We utilize the classification of all strong $(3\mod 5)$-arcs in $\PG(3,5)$ with cardinality at most $158$ not 
  containing a full hyperplane in their support. If $\widetilde{\cK}$ is lifted and $\#\widetilde{\cK}<168$, then $\widetilde{\cK}$ is lifted from a strong 
  $(3\mod 5)$-arc $\cF$ in $\PG(2,5)$ with $\#\cF\in\{18,23,28\}$. In the first case $\#\cF=18$ there is a full line, so that the lifted arc $\cK$ would contain 
  a full hyperplane in its support. In the two other cases $\cF$ contains a line of type $A_1$, so that Lemma~\ref{lemma_no_threefold_line_hyperplane} 
  gives a contradiction for $\widetilde{\cK}$. If $\widetilde{\cK}$ is non-lifted, then we have $\#\widetilde{\cK}\in\{128,143\}$ and we can apply Lemma~\ref{lemma_tilde_k_neq_128_143}. 
\end{proof}
In the following we will need a few restrictions on the spectrum of arcs in $\PG(2,5)$ that we will briefly prove for the reader's convenience. 
\begin{Lemma}
  \label{lemma_spectra_22_5_3_5_arcs_partial}
  The spectrum $\left(a_i\right)$ of a $(22,5)$-arc $\cK$ in $PG(2,5)$ satisfies 
  $a_1=0$, $a_3 = 13-10a_0-3a_2$, $a_4 = -3+15a_0+3a_2$, and $a_5 = 21-6a_0-a_2$, where $a_0\le 1$ and $a_2\le \left\lfloor(13-10a_0)/3\right\rfloor$.   
\end{Lemma}
\begin{proof}
  From Lemma~\ref{lemma_residual} and $m\le 5$ we conclude that $\cK$ is projective, i.e., $\cK(P)\in\{0,1\}$ for all $P\in\cP$. 
  Applying Lemma~\ref{lemma_residual} with $m=1$ gives a maximum point multiplicity of $0$ on this line, which is absurd, 
  so that we assume $a_1=0$ in the following. With this, the standard equations are given by
  $a_0+a_2+a_3+a_4+a_5 =  31$, $2a_2+3a_3+4a_4+5a_5 = 132$, and $a_2+3a_3+6a_4+10a_5 = 231$, so that $a_3 = 13-10a_0-3a_2$, 
  $a_4 = -3+15a_0+3a_2$,  and $a_5 = 21-6a_0-a_2$. Since $a_3\ge 0$ and $a_0,a_2\in\mathbb{N}$, we have $a_0\le 1$ and 
  $a_2\le \left\lfloor(13-10a_0)/3\right\rfloor$.
\end{proof}
An important implication is that every $22$-plane in a $(104,22)$-arc in $\PG(3,5)$ does not contain a $1$-line.
\begin{Lemma}
  \label{lemma_spectra_6_2_3_5_arcs}
  The spectrum $\left(a_i\right)$ of a $(6,2)$-arc $\cK$ in $\PG(2,5)$-arc satisfies $a_0=10$, $a_1=6$, and $a_2=15$.
\end{Lemma}  
\begin{proof}
  From Lemma~\ref{lemma_residual} and $m\le 2$ we conclude that $\cK$ is projective, i.e., $\cK(P)\in\{0,1\}$ for all $P\in\cP$. 
  With this, the standard equations are given by
  $a_0+a_1+a_2 = 31$, $a_1+2a_2 = 36$, and $a_2+3a_2 = 15$, yielding the stated unique solution.
\end{proof}

\begin{Lemma}
  \label{lemma_spectra_9_3_3_5_arcs_partial}
  The spectrum $\left(a_i\right)$ of a $(9,3)$-arc $\cK$ in $\PG(2,5)$ satisfies 
  $a_0 = 13 -a_3$, $a_1 = -18 +3a_3$, and $a_2 = 36 -3a_3$, where $6\le a_3\le 12$.   
\end{Lemma}
\begin{proof}
  From Lemma~\ref{lemma_residual} and $m\le 3$ we conclude that $\cK$ is projective, i.e., $\cK(P)\in\{0,1\}$ for all $P\in\cP$.  
  With this, the standard equations are given by
  $a_0+a_1+a_2+a_3=31$, $a_1+2a_2+3a_3   = 54$, and $a_2+3a_3 = 36$, 
  so that $a_0 = 13 -a_3$, $a_1 = -18 +3a_3$, and $a_2 = 36 -3a_3$. Since $a_1\ge 0$ and $a_2\ge 0$, we have $6\le a_3\le 12$.
\end{proof}
Note that the cases $a_3\in\{6,11,12\}$ in Lemma~\ref{lemma_spectra_9_3_3_5_arcs_partial} cannot occur. However, we will not need this extra information. 

\begin{Lemma}
  \label{lemma_spectra_10_3_3_5_arcs_partial}
  The spectrum $\left(a_i\right)$ of a $(10,3)$-arc $\cK$ in $\PG(2,5)$ satisfies 
  $a_0 = 16 -a_3$, $a_1 = -30 +3a_3$, and $a_2 = 45 -3a_3$, where $10\le a_3\le 15$.   
\end{Lemma}
\begin{proof}
  From Lemma~\ref{lemma_residual} and $m\le 3$ we conclude that $\cK$ is projective, i.e., $\cK(P)\in\{0,1\}$ for all $P\in\cP$.  
  With this, the standard equations are given by
  $a_0+a_1+a_2+a_3 = 31$, $a_1+2a_2+3a_3 = 60$, and $a_2+3a_3 = 45$, so that $a_0 = 16 -a_3$, $a_1 = -30 +3a_3$, and $a_2 = 45 -3a_3$.   
  Since $a_1\ge 0$ and $a_2\ge 0$, we have $10\le a_3\le 15$.
\end{proof}

\begin{Lemma}
  \label{lemma_spectra_11_3_3_5_arcs_partial}
  The spectrum $\left(a_i\right)$ of an $(11,3)$-arc $\cK$ in $\PG(2,5)$ satisfies 
  $a_0 = 20 -a_3$, $a_1 = -44 +3a_3$, and $a_2 = 55 -3a_3$, where $15\le a_3\le 18$.   
\end{Lemma}
\begin{proof}
  From Lemma~\ref{lemma_residual} and $m\le 3$ we conclude that $\cK$ is projective, i.e., $\cK(P)\in\{0,1\}$ for all $P\in\cP$.  
  With this, the standard equations are given by
  $a_0+a_1+a_2+a_3= 31$, $a_1+2a_2+3a_3 = 66$, and $a_2+3a_3 = 55$, so that $a_0 = 20 -a_3$, $a_1 = -44 +3a_3$, and $a_2 = 55 -3a_3$. 
  Since $a_1\ge 0$ and $a_2\ge 0$, we have $15\le a_3\le 18$.
\end{proof}

\begin{Lemma}
  \label{lemma_104_22_3_5_a1}
  Let $\cK$ be a $(104,22)$-arc in $\PG(3,5)$. Then $a_1=0$.
\end{Lemma}
\begin{proof}
  Assume that $H_0$ is a $1$-plane and consider a $1$-line $L$ in $H_0$. By $H_1,\dots,H_5$ we denote the other $5$ planes through $L$.  
  From Lemma~\ref{lemma_spectra_22_5_3_5_arcs_partial} we conclude $\cK(H_i)\le 21$ for all $1\le i\le 5$, so that 
  $\#\cK =\sum_{i=0}^5 \cK(H_i)-5\cdot\cK(L)\le 101<104$, which is a contradiction.
\end{proof}

The following implication of the standard equations will be important in the remaining part.
\begin{Lemma}
  \label{lemma_hyperplane_contribution}
  The spectrum $\left(a_i\right)$ of an $(n,s)$-arc $\cK$ in $\PG(k-1,q)$ satisfies
  $$
    \sum_{H\in\cH} { {s-\cK(H)} \choose 2} = {s\choose 2}\cdot[k]_q -n(s-1)\cdot[k-1]_q+{n\choose 2}\cdot[k-2]_q+q^{k-2}\cdot\sum_{i\ge 2} {i\choose 2}\lambda_i.
  $$
\end{Lemma}
\begin{proof}
  From the standard equations we conclude 
  $$
    \sum_{i=0}^s {{s-i}\choose 2} a_i= {s\choose 2}\cdot[k]_q -n(s-1)\cdot[k-1]_q+{n\choose 2}\cdot[k-2]_q+q^{k-2}\cdot\sum_{i\ge 2} {i\choose 2}\lambda_i  
  $$
  and replace the left-hand side by $\sum_{H\in\cH} { {s-\cK(H)} \choose 2}$.
\end{proof}
The idea is to use some information on $\widetilde{\cK}$ to bound the left hand side of the equation in Lemma~\ref{lemma_hyperplane_contribution}. So, for a given 
$(n,s)$-arc $\cK$ in $\PG(k-1,q)$, where $k\ge 3$ and $H_0$ is a fixed hyperplane, we denote by $H_1(S),\dots,H_q(S)$ the $q$  other hyperplanes through $S$ 
and set     
\begin{equation}
  \label{eq_eta_i_j}
  \eta_{i,j}(H_0)=\underset{S\,:\,\cK(S)=i,\tilde{\cK}(\tilde{S})=j,S\le H_0,\dim(S)=k-2}{\max} \,\,\sum_{h=1}^q {{w-\cK(H_h(S))}\choose 2}.
\end{equation}  
If there exists no hyperline $S$ with $\cK(S)=i$ or $\tilde{\cK}(\tilde{S})=j$, then we set $\eta_{i,j}=0$. 
We abbreviate $\eta_{i,j}(H_0)$ as $\eta_{i,j}$ whenever $H_0$ is clear from the context. With this and $\sum_{i\ge 2} {i\choose 2}\lambda_i\ge 0$ we 
directly obtain: 
\begin{Lemma}
  \label{lemma_hyperplane_contribution_refined}
  Let $\cK$ be an $(n,s)$-arc in $\PG(k-1,q)$, where $k\ge 3$, $H_0$ be a hyperplane,  
  $b_{i,j}$ be the number of hyperlines $S$ in $H_0$ with $\cK(S)=i$ and $\tilde{\cK}(\tilde{S})=j$ of the restriction $\cK|_{H_0}$,   
  and $\widehat{\eta}_{i,j}$ some numbers satisfying $\eta_{i,j}\le\widehat{\eta}_{i,j}$ for all $i,j\in\N_0$. Then, we have
  \begin{eqnarray}
    \sum_{i,j} b_{i,j}\widehat{\eta}_i + {{s-\cK(H_0)}\choose 2} &\!\!\!\ge\!\!\!& {s\choose 2}\cdot[k]_q-n(s\!-\!1)[k\!-1\!]_q+{n\choose 2}\cdot[k\!-\!2]_q.\label{ie_hyperplane_contribution_refined}
  \end{eqnarray}   
\end{Lemma}

Plugging in our specific data $k=4$, $n=104$, $s=22$, and $q=5$, into Inequality~(\ref{ie_hyperplane_contribution_refined}) gives
\begin{equation}
  \label{ie_no_104_22_main_1}
  \sum_{i,j} b_{i,j}\widehat{\eta}_{i,j} \,+\, {{22-\cK(H_0)}\choose 2} \ge 468. 
\end{equation}
Summing up the multiplicities of the lines $\tilde{L}$ through $\tilde{H}_0$ gives
\begin{equation}
  \label{ie_no_104_22_main_2}
  \#\tilde{\cK}=\tilde{\cK}\!\left(\tilde{H}_0\right)+\sum_{i,j} b_{i,j}\left(j-\tilde{\cK}\!\left(\tilde{H}_0\right)\right)\ge 163,
\end{equation}
taking Lemma~\ref{lemma_lower_bound_tilde_k} into account. The strategy of the remaining argumentation is the following. We pick a not excluded possibility 
for the multiplicity $\cK(H_0)$ of a hyperplane $H_0$ and determine some information on the spectrum $\left(b_i\right)$ of $\cK|_{H_0}$ and compute 
values $\widehat{\eta}_{i,j}$ based on the current knowledge of the possible hyperplane multiplicities with respect to $\cK$. Surely, the unknown 
values $b_{i,j}\in\N_0$ are linked to the $b_i$ via
$$
  \sum_j b_{i,j}=b_i
$$
for all $i\in \N_0$. Then we will show that Inequality~(\ref{ie_no_104_22_main_1}) and Inequality~(\ref{ie_no_104_22_main_2}) cannot be satisfied simultaneously. 

In the following lemmas we always start with a hyperplane $H_0$ of a $(104,22)$-arc $\cK$ in $\PG(3,5)$. For an arbitrary fixed line $L$ in $H_0$ we denote by  
$H_1(L),\dots,H_5(L)$ the other $5$ planes through $L$. For brevity, we write $H_i$ instead of $H_i(L)$, where $1\le i\le 5$.  
\begin{Lemma}
  \label{lemma_104_22_3_5_a0}
  Let $\cK$ be a $(104,22)$-arc in $\PG(3,5)$. Then $a_0=0$.
\end{Lemma}
\begin{proof}
  Let $H_0$ be a $0$-plane, so that $\cK(L)=0$ for each line $L$ in $H_0$. Looping over all possibilities, while taking into account 
  $\cK(H_i)\in\{0,4,5,6,9,10,11,14,15,16,19,20,21,22\}$, we compute the values of $\sum_{i=1}^{5} {{22-\cK(H_i)}\choose 2}$ as follows:
  \begin{center}
    \begin{tabular}{|c|c|c|c|c|}
      \hline
      $\cK(L)$ & $\tilde{\cK}(\tilde{L})$ & $\left(\cK(H_0),\dots,\cK(H_5)\right)$ & $\sum_{i=1}^{5} {{22-\cK(H_i)}\choose 2}$ & type of $\tilde{L}$\\
      \hline
      0 &  3 & $(0,22,22,22,22,16)$ & 15 & $A_2$ \\
      \hline
      0 &  8 & $(0,21,21,21,21,20)$ & 1  & $B_8$ \\
        &    & $(0,22,21,21,20,20)$ & 2  & $B_7$ \\
        &    & $(0,22,21,21,21,19)$ & 3  & $B_4$ \\
        &    & $(0,22,22,20,20,20)$ & 3  & $B_6$ \\ 
        &    & $(0,22,22,21,20,19)$ & 4  & $B_3$ \\ 
        &    & $(0,22,22,22,19,19)$ & 6  & $B_1$ \\     
      \hline   
    \end{tabular}
  \end{center}  
  We can condense this information to the following non-zero upper bounds for $\widehat{\eta}_{i,j}$:  
  \begin{center}
    \begin{tabular}{|c|c|c|c|}
      \hline
      $i=\cK(L)$ & $j=\tilde{\cK}(\tilde{L})$ & $\widehat{\eta}_{i,j}$ & $\left(\cK(H_0),\dots,\cK(H_5)\right)$ \\
      \hline   
      0 &  3 & 15 & $(0,22,22,22,22,16)$ \\
      0 &  8 &  6 & $(0,22,22,22,19,19)$ \\     
      \hline
    \end{tabular}    
  \end{center}
  Denote by $x$ the number of lines $L$ in $H_0$ such that $\tilde{\cK}(\tilde{L})=3$. With this, we have $b_{0,3}=x$ and $b_{0,8}=31-x$. With this, 
  Inequality~(\ref{ie_no_104_22_main_1})  gives   
  $$
    x\cdot 15+(31-x)\cdot 6\,+\, {22\choose 2} \,\,\ge\,\, 468,
  $$
  so that $x\ge \left\lceil\frac{51}{9}\right\rceil=6$. 
  Using $\widetilde{\cK}(\widetilde{H}_0)=2$ Inequality~(\ref{ie_no_104_22_main_2}) yields
  $$
    \#\widetilde{K}=2+x\cdot 1+(31-x)\cdot 6=188-5x\le 158 <163,
  $$
  which is a contradiction.
\end{proof}
In the following lemmas we will not list the values $\sum_{i=1}^{5} {{22-\cK(H_i)}\choose 2}$ for all possibilities but just the resulting  
non-zero upper bounds for $\widehat{\eta}_{i,j}$.

\begin{Lemma}
  \label{lemma_104_22_3_5_a4}
  Let $\cK$ be a $(104,22)$-arc in $\PG(3,5)$. Then $a_4=0$.
\end{Lemma}
\begin{proof}
  Let $H_0$ be a $4$-plane. From Lemma~\ref{lemma_residual} we conclude $\cK(L)\le 2$ for each line $L$ in $H_0$. Looping over all 
  possibilities, while taking into account $$\cK(H_i)\in\{4,5,6,9,10,11,14,15,16,19,20,21,22\}$$ and that a $22$-plane cannot contain a $1$-line, we compute 
  the following non-zero upper bounds for $\widehat{\eta}_{i,j}$:
  \begin{center}
    \begin{tabular}{|c|c|c|c|}
      \hline
      $i=\cK(L)$ & $j=\tilde{\cK}(\tilde{L})$ & $\widehat{\eta}_{i,j}$ & $\left(\cK(H_0),\dots,\cK(H_5)\right)$ \\
      \hline
      2 & 3 &  0 & $(4,22,22,22,22,22)$ \\
      \hline
      1 & 8 &  0 & $ (4,21,21,21,21,21)$ \\
      \hline   
      0 &  8 & 29 & $(4,22,22,22,20,14)$ \\
      0 & 13 &  9 & $(4,22,21,19,19,19)$ \\     
      \hline
    \end{tabular}    
  \end{center}
  Denote by $x$ the number of lines $L$ in $H_0$ such that $\cK(L)=0$ and $\tilde{\cK}(\tilde{L})=8$. Note that $\cK|_{H_0}$ is a $(4,2)$-arc in 
  $\PG(2,5)$ with spectrum $b_0=13$, $b_1=12$, $b_2=6$, so that $b_{2,3}=6$,\footnote{An $(n,2)$-arc in $\PG(2,q)$ has spectrum $b_2={n\choose 2}$, $b_1=n\cdot (q+2-n)$, $b_0=q^2+q+1-b_1-b_2$.} $b_{1,8}=12$, $b_{0,8}=x$, and $b_{0,13}=13-x$. With this, 
  Inequality~(\ref{ie_no_104_22_main_1}) reads   
  $$
    6\cdot 0+ 12\cdot 0 x\cdot 29+(13-x)\cdot 9\,+\, {18\choose 2} \,\,\ge\,\, 468,
  $$
  so that $x\ge \left\lceil\tfrac{99}{10}\right\rceil=10$. Using $\widetilde{\cK}(\widetilde{H}_0)=3$ this contradicts Inequality~(\ref{ie_no_104_22_main_2}) since
  $$
    \#\widetilde{K}=3+6\cdot 0+ 12\cdot 5+ x\cdot 5+(13-x)\cdot 10=193-5x\le 143<163.
  $$
\end{proof}

\begin{Lemma}
  \label{lemma_104_22_3_5_a5}
  Let $\cK$ be a $(104,22)$-arc in $\PG(3,5)$. Then $a_5=0$.
\end{Lemma}
\begin{proof}
  Let $H_0$ be a $5$-plane. From Lemma~\ref{lemma_residual} we conclude $\cK(L)\le 2$ for each line $L$ in $H_0$. Looping over all 
  possibilities, while taking into account $$\cK(H_i)\in\{5,6,9,10,11,14,15,16,19,20,21,22\}$$ and that a $22$-plane cannot contain a $1$-line, we compute 
  the following non-zero upper bounds for $\widehat{\eta}_{i,j}$:
  \begin{center}
    \begin{tabular}{|c|c|c|c|}
      \hline
      $i=\cK(L)$ & $j=\tilde{\cK}(\tilde{L})$ & $\widehat{\eta}_{i,j}$ & $\left(\cK(H_0),\dots,\cK(H_5)\right)$ \\
      \hline
      2 &  3 &  0 & $(5,22,22,22,22,21)$ \\
      \hline
      1 &  8 &  1 & $(5,21,21,21,21,20)$ \\
      \hline
      0 &  3 & 55 & $(5,22,22,22,22,11)$ \\    
      0 &  8 & 31 & $(5,22,22,22,19,14)$ \\
      0 & 13 & 10 & $(5,22,20,19,19,19)$ \\     
      \hline
    \end{tabular}    
  \end{center}

  Denote by $x$ the number of lines $L$ in $H_0$ such that $\cK(L)=0$ and $\tilde{\cK}(\tilde{L})=3$ and by $y$ the number of lines $L$ in $H_0$ 
  such that $\cK(L)=0$ and $\tilde{\cK}(\tilde{L})=8$. Note that $\cK|_{H_0}$ is a $(5,2)$-arc in $\PG(2,5)$ with spectrum $a_0=11$, $a_1=10$, $a_2=10$, 
  so that $b_{2,3}=10$, $b_{1,8}=10$, $b_{0,3}=x$, $b_{0,8}=y$, and $b_{0,11}=11-x-y$. With this, 
  Inequality~(\ref{ie_no_104_22_main_1}) reads   
  $$
    10\cdot 0+ 10\cdot 1 +x\cdot 55+y\cdot 31+(11-x-y)\cdot 10\,+\, {17\choose 2} \,\,\ge\,\, 468,
  $$
  so that $45x+21y \ge 212$, which implies $90x+45y\ge 90x+42y\ge 424$. Thus, we have $2x+y\ge 10$. 
  Combining this with $\widetilde{\cK}(\widetilde{H}_0)=2$, Inequality~(\ref{ie_no_104_22_main_2}) yields the contradiction  
  $$
    \#\widetilde{K}=2+10\cdot 1+ 10\cdot 6+ x\cdot 1+y\cdot 6+(11-x-y)\cdot 11=193-10x-5y\le 143<163.
  $$
\end{proof}

\begin{Lemma}
  \label{lemma_104_22_3_5_a6}
  Let $\cK$ be a $(104,22)$-arc in $\PG(3,5)$. Then $a_6=0$.
\end{Lemma}
\begin{proof}
  Let $H_0$ be a $6$-plane. From Lemma~\ref{lemma_residual} we conclude $\cK(L)\le 2$ for each line $L$ in $H_0$. Looping over all 
  possibilities, while taking into account $$\cK(H_i)\in\{6,9,10,11,14,15,16,19,20,21,22\}$$ and that a $22$-plane cannot contain a $1$-line, we compute 
  the following non-zero upper bounds for $\widehat{\eta}_{i,j}$:  
  \begin{center}
    \begin{tabular}{|c|c|c|c|}
      \hline
      $i=\cK(L)$ & $j=\tilde{\cK}(\tilde{L})$ & $\widehat{\eta}_{i,j}$ & $\left(\cK(H_0),\dots,\cK(H_5)\right)$ \\
      \hline   
      2 &  3 &  1 & $(6,22,22,22,22,20)$ \\   
      \hline
      1 &  8 &  3 & $(6,21,21,21,21,19)$ \\ 
      \hline 
      0 &  3 & 66 & $(6,22,22,22,22,10)$ \\
      0 &  8 & 31 & $(6,22,22,21,19,14)$ \\     
      0 & 13 & 12 & $(6,22,19,19,19,19)$ \\ 
      \hline
    \end{tabular}    
  \end{center}
  
  Denote by $x$ the number of lines $L$ in $H_0$ such that $\cK(L)=0$ and $\tilde{\cK}(\tilde{L})=3$ and by $y$ the number of lines $L$ in $H_0$ such that $\cK(L)=0$ and 
  $\tilde{\cK}(\tilde{L})=8$.
   
  From Lemma~\ref{lemma_residual} and the non-existence of $(6,1)$-arcs in $\PG(2,5)$ we conclude that the restricted arc $\cK|_{H_0}$ is a $(6,2)$-arc in $\PG(2,5)$. 
  Let $\left(b_i\right)$ be the spectrum of $\cK|_{H_0}$. Given the above enumeration of the possible combinations of $i=\cK(L)$ and  $j=\tilde{\cK}(\tilde{L})$ we obtain 
  $b_{2,3}  = b_2$, $b_{1,8} = b_1$, $b_{0,3}=x$, $b_{0,8} = y$, and $b_{0,13} = b_0-x-y$, 
  so that Inequality~(\ref{ie_no_104_22_main_1}) reads    
  \begin{equation}
    \label{eq_104_22_3_5_a6_1}
    b_2\cdot 1+b_1\cdot 3+x\cdot 66+y\cdot 31+\left(b_0-x-y\right)\cdot 12\,+\, {16\choose 2} \,\,\ge\,\, 468
  \end{equation}
  and combining $\widetilde{\cK}(\widetilde{H}_0)=1$ with Inequality~(\ref{ie_no_104_22_main_2}) gives
  \begin{equation}
    \label{eq_104_22_3_5_a6_2}
    \#\widetilde{\cK}=1+b_2\cdot 2+b_1\cdot 7+x\cdot 2+y\cdot 7+\left(b_0-x-y\right)\cdot 12\ge 163.
  \end{equation}
  Plugging in $b_0=10$, $b_1=6$, and $b_2=15$, see Lemma~\ref{lemma_spectra_6_2_3_5_arcs}, into Inequality~(\ref{eq_104_22_3_5_a6_1}) and  
  Inequality~(\ref{eq_104_22_3_5_a6_2}) gives
  \begin{equation}
    \label{eq_104_22_3_5_a6_3}
    54x+19y \ge 195
  \end{equation}
  and
  $$
    \#\widetilde{\cK} = 193-10x-5y\ge 163,
  $$ 
  respectively. The latter constraint yields $2x+y\le 6$, so that
  $$
    54x+19y \le 27(2x+y)\le  162,
  $$
  which contradicts Inequality~(\ref{eq_104_22_3_5_a6_3}).
\end{proof}
Note that our application of Inequality~(\ref{eq_104_22_3_5_a6_2}) differs from the one in the proof of \cite[Lemma 4.4]{landjev2016non} due to a typo; the 
approach however is essentially the same.

\begin{Lemma}
  \label{lemma_104_22_3_5_a9}
  Let $\cK$ be a $(104,22)$-arc in $\PG(3,5)$. Then $a_9=0$.
\end{Lemma}
\begin{proof}
  Let $H_0$ be a $9$-plane. From Lemma~\ref{lemma_residual} we conclude $\cK(L)\le 3$ for each line $L$ in $H_0$. Taking into account 
  $\cK(H_i)\in\{9,10,11,14,15,16,19,20,21,22\}$ and that a $22$-plane cannot contain a $1$-line,  
  we compute the following non-zero upper bounds for $\widehat{\eta}_{i,j}$:
  \begin{center}
    \begin{tabular}{|c|c|c|c|}
      \hline
      $i=\cK(L)$ & $j=\tilde{\cK}(\tilde{L})$ & $\widehat{\eta}_{i,j}$ & $\left(\cK(H_0),\dots,\cK(H_5)\right)$ \\
      \hline   
      3 &  3 &  0 & $(9,22,22,22,22,22)$ \\ 
      \hline 
      2 &  8 &  4 & $(9,22,22,22,20,19)$ \\
      \hline
      1 &  8 & 15 & $(9,21,21,21,21,16)$ \\ 
      1 & 13 &  7 & $(9,21,21,20,19,19)$ \\
      \hline 
      0 &  8 & 79 & $(9,22,22,22,20,9)$ \\
      0 & 13 & 34 & $(9,22,21,19,19,14)$ \\     
      0 & 18 & 15 & $(9,19,19,19,19,19)$ \\
      \hline
    \end{tabular}    
  \end{center}
  Denote by $x$ the number of lines $L$ in $H_0$ such that $\cK(L)=1$ and $\tilde{\cK}(\tilde{L})=8$. Similarly, denote by $u$, resp.~$v$, the number of lines 
  $L$ in $H_0$ with $\cK(L)=0$, $\tilde{\cK}(\tilde{L})=8$, resp.~$\cK(L)=0$, $\tilde{\cK}(\tilde{L})=13$.
  
  From Lemma~\ref{lemma_residual} and the non-existence of $(9,2)$-arcs in $\PG(2,5)$ we conclude that the restricted arc $\cK|_{H_0}$ is a $(9,3)$-arc in $\PG(2,5)$. 
  Let $\left(b_i\right)$ be the spectrum of $\cK|_{H_0}$. Given the above enumeration of the possible combinations of $i=\cK(L)$ and  $j=\tilde{\cK}(\tilde{L})$ we obtain 
  $b_{3,3}  = b_3,$, $b_{2,8} = b_2$, $b_{1,8} = x$, $b_{1,13} = b_1-x$, $b_{0,8}  = u$, $b_{0,13} = v$, and $b_{0,18} = b_0-u-v$,   
  so that Inequality~(\ref{ie_no_104_22_main_1}) reads
  \begin{equation}
    \label{eq_104_22_3_5_a9_1}
    b_3\cdot 0+b_2\cdot 4+x\cdot 15+\left(b_1-x\right)\cdot 7+u\cdot 79+v\cdot 34+\left(b_0-u-v\right)\cdot 15\,+\, {13\choose 2} \,\,\ge\,\, 468.
  \end{equation}
  Using $\widetilde{\cK}(\widetilde{H}_0)=3$ Inequality~(\ref{ie_no_104_22_main_2}) gives
  \begin{eqnarray}
    \#\widetilde{\cK} &=&3+b_3\cdot 0+b_2\cdot 5+x\cdot 5+\left(b_1-x\right)\cdot 10 \notag\\ &&+u\cdot 5+v\cdot 10+\left(b_0-u-v\right)\cdot 15\ge 163.\label{eq_104_22_3_5_a9_2}
  \end{eqnarray}
  Plugging in the parameterization from Lemma~\ref{lemma_spectra_9_3_3_5_arcs_partial} into Inequality~(\ref{eq_104_22_3_5_a9_1}) and  
  Inequality~(\ref{eq_104_22_3_5_a9_2}) gives
  \begin{equation}
    \label{eq_104_22_3_5_a9_3}
    8x+64u+19v \ge 177+6b_3\ge 213
  \end{equation}
  and
  $$
    \#\widetilde{\cK} = 198-5x-10u-5v\ge 163,
  $$
  respectively. The latter constraint yields $x+2u+v\le 7$, so that $u\le 3$. Using $x+v\le 7-2u$ we conclude 
  $$
    8x+64u+19v \le 19\cdot (7-2u)+64u=133+26u\le 211
  $$ 
  from $u\le 3$, which contradicts Inequality~(\ref{eq_104_22_3_5_a9_3}).  
\end{proof}

\begin{Lemma}
  \label{lemma_104_22_3_5_a10}
  Let $\cK$ be a $(104,22)$-arc in $\PG(3,5)$. Then $a_{10}=0$.
\end{Lemma}
\begin{proof}
  Let $H_0$ be a $10$-plane. From Lemma~\ref{lemma_residual} we conclude $\cK(L)\le 3$ for each line $L$ in $H_0$. Looping over all 
  possibilities, while taking into account $$\cK(H_i)\in\{10,11,14,15,16,19,20,21,22\}$$ and that a $22$-plane cannot contain a $1$-line, we compute 
  the following non-zero upper bounds for $\widehat{\eta}_{i,j}$:
  \begin{center}
    \begin{tabular}{|c|c|c|c|}
      \hline
      $i=\cK(L)$ & $j=\tilde{\cK}(\tilde{L})$ & $\widehat{\eta}_{i,j}$ & $\left(\cK(H_0),\dots,\cK(H_5)\right)$ \\
      \hline   
      3 &  3 &  0 & $(10,22,22,22,22,21)$ \\ 
      \hline
      2 &  3 & 15 & $(10,22,22,22,22,16)$ \\   
      2 &  8 &  6 & $(10,22,22,22,19,19)$ \\
      \hline
      1 &  8 & 21 & $(10,21,21,21,21,15)$ \\ 
      1 & 13 &  9 & $(10,21,21,19,19,19)$ \\
      \hline 
      0 &  8 & 69 & $(10,22,22,21,19,10)$ \\     
      0 & 13 & 35 & $(10,22,20,19,19,14)$ \\
      \hline
    \end{tabular}    
  \end{center}

  Denote by $x$ the number of lines $L$ in $H_0$ such that $\cK(L)=2$ and $\tilde{\cK}(\tilde{L})=3$, by $y$ the number of lines $L$ in $H_0$ such that $\cK(L)=1$ and 
  $\tilde{\cK}(\tilde{L})=8$, and by $z$ the number of lines $L$ in $H_0$ such that $\cK(L)=0$ and $\tilde{\cK}(\tilde{L})=8$.
   
  From Lemma~\ref{lemma_residual} and the non-existence of $(10,2)$-arcs in $\PG(2,5)$ we conclude that the restricted arc $\cK|_{H_0}$ is a $(10,3)$-arc in $\PG(2,5)$. 
  Let $\left(b_i\right)$ be the spectrum of $\cK|_{H_0}$. Given the above enumeration of the possible combinations of $i=\cK(L)$ and  $j=\tilde{\cK}(\tilde{L})$ we obtain 
  $b_{3,3}  = b_3$, $b_{2,3} = x$, $b_{2,8}=b_2-x$, $b_{1,8} = y$, $b_{1,13} = b_1-y$, $b_{0,8}  = z$, and $b_{0,13} = b_0-z$, 
  so that Inequality~(\ref{ie_no_104_22_main_1}) reads   
  \begin{equation}
    \label{eq_104_22_3_5_a10_1}
    b_3\cdot 0+x\cdot 15+\left(b_2-x\right)\cdot 6+y\cdot 21+\left(b_1-y\right)\cdot9 +z\cdot 69+\left(b_0-z\right)\cdot 35\,+\, {12\choose 2} \,\,\ge\,\, 468.
  \end{equation}
  Using $\widetilde{\cK}(\widetilde{H}_0)=2$ Inequality~(\ref{ie_no_104_22_main_2}) gives
  \begin{equation}
    \label{eq_104_22_3_5_a10_2}
    \#\widetilde{\cK}=2+b_3\cdot 1+x\cdot 1+\left(b_2-x\right)\cdot 6+y\cdot 6+\left(b_1-y\right)\cdot 11+z\cdot 6+\left(b_0-z\right)\cdot 11\ge 163.
  \end{equation}
  Plugging in the parameterization from Lemma~\ref{lemma_spectra_10_3_3_5_arcs_partial} into Inequality~(\ref{eq_104_22_3_5_a10_1}) and  
  Inequality~(\ref{eq_104_22_3_5_a10_2}) gives
  \begin{equation}
    \label{eq_104_22_3_5_a10_3}
    9x+12y+34z \ge 26b_3-158 
  \end{equation}
  and
  $$
    \#\widetilde{\cK} = 118-5x-5y-5z+5b_3\ge 163,
  $$ 
  respectively. The latter constraint yields $x+y+z\le b_3-9$, so that
  $$
    9x+12y+34z \le 34\cdot\left(b_3-9\right)=34b_3-306.  
  $$
  Thus, we can conclude $34b_3-306\ge 26b_3-158$ from Inequality~(\ref{eq_104_22_3_5_a10_3}), which is equivalent to $b_3\ge 18.5$. 
  Since we have $b_3\le 15$ due to Lemma~\ref{lemma_spectra_10_3_3_5_arcs_partial}, we obtain a contradiction.  
\end{proof}

\begin{Lemma}
  \label{lemma_104_22_3_5_a11}
  Let $\cK$ be a $(104,22)$-arc in $\PG(3,5)$. Then $a_{11}=0$.
\end{Lemma}
\begin{proof}
  Let $H_0$ be a $11$-plane. From Lemma~\ref{lemma_residual} we conclude $\cK(L)\le 3$ for each line $L$ in $H_0$. Looping over all 
  possibilities, while taking into account $$\cK(H_i)\in\{11,14,15,16,19,20,21,22\}$$ and that a $22$-plane cannot contain a $1$-line, we compute 
  the following non-zero upper bounds for $\widehat{\eta}_{i,j}$:
  \begin{center}
    \begin{tabular}{|c|c|c|c|}
      \hline
      $i=\cK(L)$ & $j=\tilde{\cK}(\tilde{L})$ & $\widehat{\eta}_{i,j}$ & $\left(\cK(H_0),\dots,\cK(H_5)\right)$ \\
      \hline   
      3 &  3 &  1 & $(11,22,22,22,22,20)$ \\ 
      \hline
      2 &  3 & 21 & $(11,22,22,22,22,15)$ \\   
      2 &  8 &  6 & $(11,22,22,21,19,19)$ \\
      \hline
      1 &  8 & 28 & $(11,21,21,21,21,14)$ \\ 
      1 & 13 & 10 & $(11,21,20,19,19,19)$ \\
      \hline 
      0 &  3 & 70 & $(11,22,22,22,16,11)$ \\
      0 &  8 & 61 & $(11,22,22,19,19,11)$ \\     
      0 & 13 & 37 & $(11,22,19,19,19,14)$ \\
      \hline
    \end{tabular}    
  \end{center}

  Denote by $x$ the number of lines $L$ in $H_0$ such that $\cK(L)=2$ and $\tilde{\cK}(\tilde{L})=3$, by $y$ the number of lines $L$ in $H_0$ such that $\cK(L)=1$ and 
  $\tilde{\cK}(\tilde{L})=8$, by $u$ the number of lines $L$ in $H_0$ such that $\cK(L)=0$ and $\tilde{\cK}(\tilde{L})=3$, and by $v$ the number of lines $L$ in $H_0$ 
  such that $\cK(L)=0$ and $\tilde{\cK}(\tilde{L})=8$.
   
  From Lemma~\ref{lemma_residual} and the non-existence of $(11,2)$-arcs in $\PG(2,5)$ we conclude that the restricted arc $\cK|_{H_0}$ is a $(11,3)$-arc in $\PG(2,5)$. 
  Let $\left(b_i\right)$ be the spectrum of $\cK|_{H_0}$. Given the above enumeration of the possible combinations of $i=\cK(L)$ and  $j=\tilde{\cK}(\tilde{L})$ we obtain 
  $b_{3,3}  = b_3$, $b_{2,3} = x$, $b_{2,8}=b_2-x$, $b_{1,8} = y$, $b_{1,13} = b_1-y$, $b_{0,3}=u$, $b_{0,8}  = v$, and $b_{0,13} = b_0-u-v$,
  so that Inequality~(\ref{ie_no_104_22_main_1}) reads   
  \begin{eqnarray}
    &&b_3\cdot 1+x\cdot 21+\left(b_2-x\right)\cdot 6+y\cdot 28+\left(b_1-y\right)\cdot 10 \notag\\ 
    &&+u\cdot 70+ v\cdot 61+\left(b_0-u-v\right)\cdot 37\,+\, {11\choose 2} \,\,\ge\,\, 468.\label{eq_104_22_3_5_a11_1}
  \end{eqnarray}
  Using $\widetilde{\cK}(\widetilde{H}_0)=1$ Inequality~(\ref{ie_no_104_22_main_2}) gives
  \begin{eqnarray}
    \#\widetilde{\cK}&=&1+b_3\cdot 2+x\cdot 2+\left(b_2-x\right)\cdot 7+y\cdot 7+\left(b_1-y\right)\cdot 12\notag\\ &&+u\cdot 2+v\cdot 7+\left(b_0-u-v\right)\cdot 12\ge 163.\label{eq_104_22_3_5_a11_2}
  \end{eqnarray}
  Plugging in the parameterization from Lemma~\ref{lemma_spectra_11_3_3_5_arcs_partial} into Inequality~(\ref{eq_104_22_3_5_a11_1}) and  
  Inequality~(\ref{eq_104_22_3_5_a11_2}) gives
  \begin{equation}
    \label{eq_104_22_3_5_a11_3}
    15x  + 18y + 33u + 24v \ge 24b_3-217 
  \end{equation}
  and
  $$
    \#\widetilde{\cK} = 98 - 5x - 5y - 10u - 5v+5b_3\ge 163,
  $$ 
  respectively. The latter constraint yields $x+y+2u+v\le b_3-13$, so that 
  $x+y+v\le b_3-13-2u$ and 
  $$
   15x  + 18y + 33u + 24v \le 33u +24\left(b_3-13-2u\right)=24b_3-15u-312.  
  $$
 Thus, we can conclude $24b_3-15u-312\ge 24b_3-217$ from Inequality~(\ref{eq_104_22_3_5_a11_3}), which is equivalent to $u\le -\frac{19}{3}$ contradicting $u\ge 0$.  
\end{proof}

\begin{Lemma}
  \label{lemma_104_22_3_5_a22}
  Let $\cK$ be a $(104,22)$-arc in $\PG(3,5)$. Then $a_{22}=0$.
\end{Lemma}
\begin{proof}
  Let $H_0$ be a $22$-plane. Looping over all possibilities, while taking into account $$\cK(H_i)\in\{14,15,16,19,20,21,22\}$$ and that a $22$-plane cannot contain a $1$-line, we compute 
  the following non-zero upper bounds for $\widehat{\eta}_{i,j}$:
  \begin{center}
    \begin{tabular}{|c|c|c|c|c|}
      \hline
      $b_{i,j}$ & $i=\cK(L)$ & $j=\tilde{\cK}(\tilde{L})$ & $\widehat{\eta}_{i,j}$ & $\left(\cK(H_0),\dots,\cK(H_5)\right)$ \\
      \hline
      $b_5$     & 5 &  3 &   3 & $(22,22,22,22,22,19)$ \\ 
      \hline  
      $x$       & 4 &  3 &  28 & $(22,22,22,22,22,14)$ \\
      $b_4-x$   & 4 &  8 &   7 & $(22,22,22,20,19,19)$ \\ 
      \hline   
      $u$       & 3 &  3 &  36 & $(22,22,22,22,16,15)$ \\
      $v$       & 3 &  8 &  32 & $(22,22,22,20,19,14)$ \\
      $b_3-u-v$ & 3 & 13 &  12 & $(22,21,19,19,19,19)$ \\
      \hline
      $y$       & 2 &  3 &  45 & $(22,22,22,16,16,16)$ \\
      $z$ & 2   &  8 &  57 & $(22,22,22,20,14,14)$ \\
      $b_2-y-z$ & 2 & 13 &  37 & $(22,21,19,19,19,14)$ \\ 
      \hline
      $b_0-s$       & 0 &  8 &  86 & $(22,22,16,16,14,14)$ \\     
      $s$   & 0 & 13 &  87 & $(22,21,19,14,14,14)$ \\
      \hline
    \end{tabular}    
  \end{center}

  From the non-existence of a $(22,4)$-arc in $\PG(2,5)$ we conclude that $\cK|_{H_0}$ is a $(22,5)$-arc in $PG(2,5)$. Let $\left(b_i\right)$ be the spectrum of $\cK|_{H_0}$. 
  Using the non-negative integer variables $x$, $y$, $z$, $u$, $v$, and $s$, we express the counts $b_{i,j}$ of the number of lines $L$ in $H_0$ such that 
  $\cK(L)=i$ and $\tilde{\cK}(\tilde{L})=j$, see the above table. With this, Inequality~(\ref{ie_no_104_22_main_1}) reads   
  \begin{eqnarray}
    b_5\cdot 3+x\cdot 28+\left(b_4-x\right)\cdot 7+u\cdot 36+v\cdot 32+\left(b_3-u-v\right)\cdot 12\notag\\ 
     +y\cdot 45+z\cdot 57+\left(b_2-y-z\right)\cdot 37+s\cdot 87+\left(b_0-s\right)\cdot 86
    \,+\, {0\choose 2} &\ge& 468.\label{eq_104_22_3_5_a22_1}
  \end{eqnarray}
  Using $\widetilde{\cK}(\widetilde{H}_0)=0$ Inequality~(\ref{ie_no_104_22_main_2}) gives
  \begin{eqnarray}
     \#\widetilde{\cK}&=&3\left(b_5+x+u+y\right)+8\left(b_4-x+v+z+b_0-s\right)\notag\\ &&+13\left(b_3-u-v+b_2-y-z+s\right)\ge 163.\label{eq_104_22_3_5_a22_2}
  \end{eqnarray}
  Plugging in the parameterization from Lemma~\ref{lemma_spectra_22_5_3_5_arcs_partial} into Inequality~(\ref{eq_104_22_3_5_a22_1}) and  
  Inequality~(\ref{eq_104_22_3_5_a22_2}) gives
  \begin{equation}
    \label{eq_104_22_3_5_a22_3}
      21x + 24u + 20v + 8y + 20z + s \ge 270 -53b_0 - 19b_2 
  \end{equation}
  and
  $$
    \#\widetilde{\cK} = 208 - 20b_0 - 5b_2 - 5x - 10u - 10y - 5v - 5z + 5s\ge 163 ,
  $$ 
  respectively. The latter constraint yields 
  $x+v+z +2u +2y  -s  \le 9 - 4b_0 - b_2$, so that $x+v+z \le 9 - 4b_0 - b_2 -2(u+y)+s$ and  
  \begin{eqnarray*}
    21x + 24u + 20v + 8y + 20z + s &\!\le\!& 21(9 - 4b_0 - b_2 -2(u\!+\!y)+s) +24(\!u+\!y)+s \\ 
    &\!=\!& 189 -84b_0-21b_2-18(u\!+\!y)+22s.
  \end{eqnarray*}
  Thus, we can conclude 
  $$
    189 -84b_0-21b_2-18(u+y)+22s \,\,\ge\,\, 270 -53b_0 - 19b_2
  $$  
  from Inequality~(\ref{eq_104_22_3_5_a22_3}), which is equivalent to
  $$
    189 +22s \,\,\ge\,\, 270 +31b_0 +2b_2+18(u+y).
  $$
  This contradicts $s\le b_0\le 1$.
\end{proof}
As a direct implication we can conclude Theorem~\ref{thm_no_104_22}. While the arguments look rather technical and lengthy, when spelled out in full, 
they actually are just an application of the linear programming method applied to Inequality~(\ref{ie_no_104_22_main_1}) and Inequality~(\ref{ie_no_104_22_main_2}).  

\section*{Acknowledgments}
The authors would like to thank the \emph{High Performance Computing group} of the University of Bayreuth for providing the excellent
computing cluster and especially Bernhard Winkler for his support. The research of the second author is supported by
the Bulgarian National Science Research Fund under Contract KP-06-N32/2 - 2019. The research of the third author was supported by the Research Fund of
Sofia University under contract No 80-10-88/25.03.2021.


\begin{thebibliography}{10}

\bibitem{baumert1973note}
L.~Baumert and R.~McEliece.
\newblock A note on the {G}riesmer bound.
\newblock {\em IEEE Transactions on Information Theory}, 19(1):134--135, 1973.

\bibitem{bouyukliev2020computer}
I.~Bouyukliev, S.~Bouyuklieva, and S.~Kurz.
\newblock Computer classification of linear codes.
\newblock {\em IEEE Transactions on Information Theory}, 18pp., to appear.

\bibitem{brouwer1997correspondence}
A.~E. Brouwer and M.~van Eupen.
\newblock The correspondence between projective codes and $2$-weight codes.
\newblock {\em Designs, Codes and Cryptography}, 11(3):261--266, 1997.

\bibitem{calderbank1986geometry}
R.~Calderbank and W.~M. Kantor.
\newblock The geometry of two-weight codes.
\newblock {\em Bulletin of the London Mathematical Society}, 18(2):97--122,
  1986.

\bibitem{dodunekov1998codes}
S.~Dodunekov and J.~Simonis.
\newblock Codes and projective multisets.
\newblock {\em The Electronic Journal of Combinatorics}, 5(1):R37, 1998.

\bibitem{griesmer1960bound}
J.~H. Griesmer.
\newblock A bound for error-correcting codes.
\newblock {\em IBM Journal of Research and Development}, 4(5):532--542, 1960.

\bibitem{hill1999extension}
R.~Hill.
\newblock An extension theorem for linear codes.
\newblock {\em Designs, Codes and Cryptography}, 17(1):151--157, 1999.

\bibitem{hill1995extensions}
R.~Hill and P.~Lizak.
\newblock Extensions of linear codes.
\newblock In {\em Proceedings of 1995 IEEE International Symposium on
  Information Theory}, page 345. IEEE, 1995.

\bibitem{kurz2020advanced}
S.~Kurz.
\newblock Lecture notes: Advanced and current topics in coding theory, February
  2020.

\bibitem{landjev2013extendability}
I.~Landjev and A.~Rousseva.
\newblock On the extendability of {G}riesmer arcs.
\newblock {\em Annual of Sofia University {\lq\lq}St. Kliment Ohridski{\rq\rq}
  -- Faculty of Mathematics and Informatics}, 101:183--192, 2013.

\bibitem{landjev2016non}
I.~Landjev and A.~Rousseva.
\newblock The non-existence of $(104, 22; 3, 5)$-arcs.
\newblock {\em Advances in Mathematics of Communications}, 10(3):601--611,
  2016.

\bibitem{landjev2017characterization}
I.~Landjev and A.~Rousseva.
\newblock On the characterization of $(3 \mod 5)$ arcs.
\newblock {\em Electronic Notes in Discrete Mathematics}, 57:187--192, 2017.

\bibitem{landjev2019divisible}
I.~Landjev and A.~Rousseva.
\newblock Divisible arcs, divisible codes, and the extension problem for arcs
  and codes.
\newblock {\em Problems of Information Transmission}, 55(3):226--240, 2019.

\bibitem{landjev2016extendability}
I.~Landjev, A.~Rousseva, and L.~Storme.
\newblock On the extendability of quasidivisible {G}riesmer arcs.
\newblock {\em Designs, Codes and Cryptography}, 79(3):535--547, 2016.

\bibitem{maruta2004new}
T.~Maruta.
\newblock A new extension theorem for linear codes.
\newblock {\em Finite Fields and Their Applications}, 10(4):674--685, 2004.

\bibitem{rousseva2015structure}
A.~Rousseva.
\newblock On the structure of $(t \mod q)$-arcs in finite projective
  geometries.
\newblock {\em Annuaire de l'Univ. de Sofia}, 102:16pp., 2015.

\bibitem{solomon1965algebraically}
G.~Solomon and J.~J. Stiffler.
\newblock Algebraically punctured cyclic codes.
\newblock {\em Information and Control}, 8(2):170--179, 1965.

\end{thebibliography}

\newpage

\textbf{Appendix}

\begin{table}[hp!]
\setlength{\tabcolsep}{0.3mm}
\begin{center}
\begin{tabular}{|c|c|c|c|c|c|c|c|c|c|c|c|c|c|c|c|c||c|c|c|c||c|}
\hline
$A_1$ & $A_2$ & $A_3$ & $B_1$ & $B_2$ & $B_3$ & $B_4$ & $B_5$ & $B_6$ & $B_7$ & $B_8$ & $C_1$ & $C_2$ & $C_3$ & $C_4$ & $C_5$ & $D_1$ & $\lambda_0$ & $\lambda_1$ & $\lambda_2$ & $\lambda_3$ & $\#$\\
\hline
0 & 12 & 16 & 0 & 0 & 0 & 0 & 0 & 0 & 0 & 3 & 0 & 0 & 0 & 0 & 0 & 0 & 16 & 12 & 3 & 0 & 1\\
3 & 0 & 25 & 0 & 0 & 0 & 0 & 3 & 0 & 0 & 0 & 0 & 0 & 0 & 0 & 0 & 0 & 15 & 15 & 0 & 1 & 1\\
4 & 25 & 0 & 0 & 0 & 0 & 0 & 1 & 0 & 0 & 0 & 0 & 0 & 0 & 0 & 1 & 0 & 20 & 5 & 5 & 1 & 1\\
30 & 0 & 0 & 0 & 0 & 0 & 0 & 0 & 0 & 0 & 0 & 0 & 0 & 0 & 0 & 0 & 1 & 25 & 0 & 0 & 6 & 1\\
\hline
\end{tabular}
\caption{Strong $(3\mod 5)$-arcs in $\PG(2,5)$ of cardinality $18$.}
\label{tab_strong_3_mod_5_arcs_pg_2_5_n_18}
\end{center}
\end{table}

\begin{table}[hp!]
\setlength{\tabcolsep}{0.3mm}
\begin{center}
\begin{tabular}{|c|c|c|c|c|c|c|c|c|c|c|c|c|c|c|c|c||c|c|c|c||c|}
\hline
$A_1$ & $A_2$ & $A_3$ & $B_1$ & $B_2$ & $B_3$ & $B_4$ & $B_5$ & $B_6$ & $B_7$ & $B_8$ & $C_1$ & $C_2$ & $C_3$ & $C_4$ & $C_5$ & $D_1$ & $\lambda_0$ & $\lambda_1$ & $\lambda_2$ & $\lambda_3$ & $\#$\\
\hline
6 & 12 & 4 & 0 & 3 & 6 & 0 & 0 & 0 & 0 & 0 & 0 & 0 & 0 & 0 & 0 & 0 & 18 & 6 & 4 & 3 & 1\\
\hline
\end{tabular}
\caption{Strong $(3\mod 5)$-arcs in $\PG(2,5)$ of cardinality $23$.}
\label{tab_strong_3_mod_5_arcs_pg_2_5_n_23}
\end{center}
\end{table}

\begin{table}[htp]
\setlength{\tabcolsep}{0.3mm}
\begin{center}
\begin{tabular}{|c|c|c|c|c|c|c|c|c|c|c|c|c|c|c|c|c||c|c|c|c||c|}
\hline
$A_1$ & $A_2$ & $A_3$ & $B_1$ & $B_2$ & $B_3$ & $B_4$ & $B_5$ & $B_6$ & $B_7$ & $B_8$ & $C_1$ & $C_2$ & $C_3$ & $C_4$ & $C_5$ & $D_1$ & $\lambda_0$ & $\lambda_1$ & $\lambda_2$ & $\lambda_3$ & $\#$\\
\hline
6 & 0 & 10 & 0 & 15 & 0 & 0 & 0 & 0 & 0 & 0 & 0 & 0 & 0 & 0 & 0 & 0 & 15 & 10 & 0 & 6 & 1\\
\hline
\end{tabular}
\caption{Strong $(3\mod 5)$-arcs in $\PG(2,5)$ of cardinality $28$.}
\label{tab_strong_3_mod_5_arcs_pg_2_5_n_28}
\end{center}
\end{table}

\begin{table}[htp]
\setlength{\tabcolsep}{0.3mm}
\begin{center}
\begin{tabular}{|c|c|c|c|c|c|c|c|c|c|c|c|c|c|c|c|c||c|c|c|c||c|}
\hline
$A_1$ & $A_2$ & $A_3$ & $B_1$ & $B_2$ & $B_3$ & $B_4$ & $B_5$ & $B_6$ & $B_7$ & $B_8$ & $C_1$ & $C_2$ & $C_3$ & $C_4$ & $C_5$ & $D_1$ & $\lambda_0$ & $\lambda_1$ & $\lambda_2$ & $\lambda_3$ & $\#$\\
\hline
0 & 0 & 10 & 0 & 15 & 0 & 0 & 6 & 0 & 0 & 0 & 0 & 0 & 0 & 0 & 0 & 0 & 10 & 15 & 0 & 6 & 1\\
0 & 3 & 7 & 2 & 8 & 2 & 7 & 1 & 0 & 0 & 1 & 0 & 0 & 0 & 0 & 0 & 0 & 11 & 12 & 3 & 5 & 1\\
0 & 6 & 4 & 0 & 6 & 12 & 0 & 0 & 0 & 0 & 3 & 0 & 0 & 0 & 0 & 0 & 0 & 12 & 9 & 6 & 4 & 1\\
0 & 6 & 4 & 2 & 4 & 8 & 4 & 0 & 0 & 2 & 1 & 0 & 0 & 0 & 0 & 0 & 0 & 12 & 9 & 6 & 4 & 2\\
0 & 6 & 4 & 3 & 3 & 6 & 6 & 0 & 0 & 3 & 0 & 0 & 0 & 0 & 0 & 0 & 0 & 12 & 9 & 6 & 4 & 1\\
0 & 9 & 1 & 3 & 0 & 9 & 3 & 0 & 3 & 3 & 0 & 0 & 0 & 0 & 0 & 0 & 0 & 13 & 6 & 9 & 3 & 1\\
2 & 8 & 1 & 8 & 6 & 4 & 0 & 0 & 0 & 1 & 0 & 0 & 0 & 0 & 1 & 0 & 0 & 15 & 5 & 5 & 6 & 1\\
4 & 5 & 2 & 5 & 4 & 10 & 0 & 0 & 0 & 0 & 0 & 1 & 0 & 0 & 0 & 0 & 0 & 15 & 5 & 5 & 6 & 1\\
8 & 4 & 0 & 16 & 0 & 0 & 0 & 0 & 1 & 0 & 0 & 2 & 0 & 0 & 0 & 0 & 0 & 18 & 1 & 4 & 8 & 1\\
\hline
\end{tabular}
\caption{Strong $(3\mod 5)$-arcs in $\PG(2,5)$ of cardinality $33$.}
\label{tab_strong_3_mod_5_arcs_pg_2_5_n_33}
\end{center}
\end{table}

\begin{table}[htp]
\setlength{\tabcolsep}{0.3mm}
\begin{center}
\begin{tabular}{|c|c|c|c|c|c|c|c|c|c|c|c|c|c|c|c|c||c|c|c|c||c|}
\hline
$A_1$ & $A_2$ & $A_3$ & $B_1$ & $B_2$ & $B_3$ & $B_4$ & $B_5$ & $B_6$ & $B_7$ & $B_8$ & $C_1$ & $C_2$ & $C_3$ & $C_4$ & $C_5$ & $D_1$ & $\lambda_0$ & $\lambda_1$ & $\lambda_2$ & $\lambda_3$ & $\#$\\
\hline
0 & 0 & 4 & 0 & 0 & 0 & 0 & 0 & 3 & 18 & 6 & 0 & 0 & 0 & 0 & 0 & 0 & 6 & 12 & 13 & 0 & 1\\
0 & 0 & 5 & 0 & 0 & 6 & 12 & 2 & 0 & 2 & 3 & 1 & 0 & 0 & 0 & 0 & 0 & 7 & 14 & 6 & 4 & 1\\
0 & 1 & 4 & 0 & 0 & 10 & 4 & 1 & 0 & 8 & 2 & 0 & 0 & 1 & 0 & 0 & 0 & 8 & 11 & 9 & 3 & 1\\
0 & 1 & 4 & 0 & 0 & 9 & 6 & 0 & 1 & 6 & 3 & 0 & 0 & 1 & 0 & 0 & 0 & 8 & 11 & 9 & 3 & 1\\
0 & 2 & 3 & 0 & 0 & 6 & 9 & 0 & 1 & 7 & 2 & 0 & 1 & 0 & 0 & 0 & 0 & 8 & 11 & 9 & 3 & 1\\
0 & 2 & 4 & 0 & 12 & 0 & 8 & 3 & 0 & 0 & 0 & 1 & 1 & 0 & 0 & 0 & 0 & 9 & 13 & 2 & 7 & 1\\
0 & 2 & 4 & 4 & 5 & 4 & 8 & 0 & 0 & 2 & 0 & 0 & 0 & 2 & 0 & 0 & 0 & 10 & 10 & 5 & 6 & 2\\
0 & 3 & 2 & 0 & 0 & 8 & 2 & 0 & 4 & 10 & 1 & 0 & 0 & 0 & 1 & 0 & 0 & 9 & 8 & 12 & 2 & 1\\
0 & 3 & 3 & 2 & 6 & 6 & 8 & 0 & 0 & 0 & 1 & 1 & 0 & 0 & 1 & 0 & 0 & 10 & 10 & 5 & 6 & 1\\
0 & 4 & 2 & 4 & 2 & 10 & 3 & 0 & 1 & 3 & 0 & 0 & 0 & 1 & 1 & 0 & 0 & 11 & 7 & 8 & 5 & 1\\
0 & 5 & 0 & 0 & 0 & 5 & 0 & 0 & 10 & 10 & 0 & 0 & 0 & 0 & 0 & 1 & 0 & 10 & 5 & 15 & 1 & 1\\
0 & 5 & 1 & 2 & 4 & 12 & 0 & 1 & 0 & 4 & 0 & 0 & 1 & 0 & 1 & 0 & 0 & 11 & 7 & 8 & 5 & 1\\
0 & 5 & 1 & 3 & 3 & 9 & 4 & 0 & 1 & 3 & 0 & 0 & 1 & 0 & 1 & 0 & 0 & 11 & 7 & 8 & 5 & 1\\
0 & 6 & 0 & 4 & 0 & 12 & 0 & 0 & 6 & 0 & 1 & 0 & 0 & 0 & 2 & 0 & 0 & 12 & 4 & 11 & 4 & 1\\
1 & 1 & 4 & 2 & 4 & 7 & 9 & 0 & 0 & 1 & 0 & 1 & 0 & 1 & 0 & 0 & 0 & 10 & 10 & 5 & 6 & 1\\
1 & 2 & 3 & 3 & 1 & 13 & 2 & 0 & 1 & 3 & 0 & 0 & 0 & 2 & 0 & 0 & 0 & 11 & 7 & 8 & 5 & 1\\
1 & 3 & 2 & 2 & 1 & 13 & 4 & 0 & 1 & 2 & 0 & 1 & 0 & 0 & 1 & 0 & 0 & 11 & 7 & 8 & 5 & 1\\
1 & 4 & 1 & 0 & 4 & 14 & 0 & 1 & 0 & 4 & 0 & 0 & 2 & 0 & 0 & 0 & 0 & 11 & 7 & 8 & 5 & 1\\
1 & 4 & 1 & 1 & 3 & 11 & 4 & 0 & 1 & 3 & 0 & 0 & 2 & 0 & 0 & 0 & 0 & 11 & 7 & 8 & 5 & 1\\
2 & 5 & 0 & 10 & 2 & 7 & 0 & 0 & 2 & 0 & 0 & 0 & 2 & 1 & 0 & 0 & 0 & 14 & 3 & 7 & 7 & 2\\
3 & 0 & 4 & 3 & 15 & 0 & 3 & 0 & 0 & 0 & 0 & 3 & 0 & 0 & 0 & 0 & 0 & 12 & 9 & 1 & 9 & 1\\
\hline
\end{tabular}
\caption{Strong $(3\mod 5)$-arcs in $\PG(2,5)$ of cardinality $38$.}
\label{tab_strong_3_mod_5_arcs_pg_2_5_n_38}
\end{center}
\end{table}

\begin{table}[htp]
\setlength{\tabcolsep}{0.3mm}
\begin{center}
\begin{tabular}{|c|c|c|c|c|c|c|c|c|c|c|c|c|c|c|c|c||c|c|c|c||c|}
\hline
$A_1$ & $A_2$ & $A_3$ & $B_1$ & $B_2$ & $B_3$ & $B_4$ & $B_5$ & $B_6$ & $B_7$ & $B_8$ & $C_1$ & $C_2$ & $C_3$ & $C_4$ & $C_5$ & $D_1$ & $\lambda_0$ & $\lambda_1$ & $\lambda_2$ & $\lambda_3$ & $\#$\\
\hline
0 & 0 & 0 & 0 & 0 & 0 & 0 & 30 & 0 & 0 & 0 & 0 & 0 & 0 & 0 & 0 & 1 & 0 & 25 & 0 & 6 & 1\\
0 & 0 & 0 & 0 & 0 & 0 & 0 & 4 & 0 & 0 & 25 & 0 & 0 & 0 & 0 & 2 & 0 & 0 & 20 & 10 & 1 & 1\\
0 & 0 & 1 & 0 & 0 & 0 & 9 & 3 & 0 & 6 & 9 & 0 & 0 & 0 & 3 & 0 & 0 & 3 & 16 & 9 & 3 & 1\\
0 & 0 & 2 & 0 & 2 & 7 & 8 & 1 & 0 & 4 & 3 & 0 & 0 & 2 & 2 & 0 & 0 & 6 & 12 & 8 & 5 & 2\\
0 & 0 & 2 & 0 & 3 & 1 & 13 & 4 & 0 & 2 & 2 & 0 & 1 & 3 & 0 & 0 & 0 & 5 & 15 & 5 & 6 & 1\\
0 & 0 & 3 & 2 & 8 & 5 & 6 & 1 & 0 & 0 & 1 & 1 & 0 & 4 & 0 & 0 & 0 & 8 & 11 & 4 & 8 & 1\\
0 & 0 & 3 & 4 & 6 & 0 & 12 & 0 & 1 & 0 & 0 & 1 & 0 & 4 & 0 & 0 & 0 & 8 & 11 & 4 & 8 & 1\\
0 & 1 & 0 & 0 & 0 & 0 & 8 & 0 & 0 & 12 & 7 & 0 & 0 & 0 & 1 & 2 & 0 & 4 & 13 & 12 & 2 & 1\\
0 & 1 & 1 & 0 & 2 & 3 & 13 & 0 & 1 & 3 & 3 & 0 & 1 & 1 & 2 & 0 & 0 & 6 & 12 & 8 & 5 & 1\\
0 & 1 & 1 & 0 & 2 & 4 & 11 & 1 & 0 & 5 & 2 & 0 & 1 & 1 & 2 & 0 & 0 & 6 & 12 & 8 & 5 & 2\\
0 & 1 & 1 & 0 & 2 & 8 & 4 & 0 & 2 & 7 & 2 & 0 & 0 & 0 & 4 & 0 & 0 & 7 & 9 & 11 & 4 & 2\\
0 & 1 & 2 & 1 & 9 & 4 & 7 & 1 & 0 & 0 & 1 & 1 & 1 & 3 & 0 & 0 & 0 & 8 & 11 & 4 & 8 & 1\\
0 & 1 & 2 & 6 & 0 & 12 & 0 & 0 & 3 & 2 & 0 & 0 & 0 & 2 & 3 & 0 & 0 & 10 & 5 & 10 & 6 & 1\\
0 & 2 & 0 & 0 & 0 & 12 & 0 & 0 & 4 & 8 & 1 & 0 & 0 & 1 & 0 & 3 & 0 & 8 & 6 & 14 & 3 & 1\\
0 & 2 & 0 & 0 & 1 & 7 & 7 & 0 & 1 & 8 & 1 & 0 & 1 & 0 & 2 & 1 & 0 & 7 & 9 & 11 & 4 & 1\\
0 & 2 & 0 & 0 & 1 & 8 & 5 & 1 & 0 & 10 & 0 & 0 & 1 & 0 & 2 & 1 & 0 & 7 & 9 & 11 & 4 & 1\\
0 & 2 & 1 & 0 & 10 & 3 & 8 & 1 & 0 & 0 & 1 & 1 & 2 & 2 & 0 & 0 & 0 & 8 & 11 & 4 & 8 & 1\\
0 & 2 & 1 & 1 & 8 & 2 & 11 & 1 & 0 & 0 & 0 & 2 & 1 & 1 & 1 & 0 & 0 & 8 & 11 & 4 & 8 & 1\\
0 & 2 & 1 & 2 & 4 & 11 & 5 & 0 & 0 & 0 & 1 & 2 & 0 & 0 & 3 & 0 & 0 & 9 & 8 & 7 & 7 & 1\\
0 & 2 & 1 & 2 & 5 & 10 & 4 & 0 & 0 & 1 & 1 & 1 & 1 & 1 & 2 & 0 & 0 & 9 & 8 & 7 & 7 & 1\\
0 & 2 & 1 & 2 & 6 & 9 & 3 & 0 & 0 & 2 & 1 & 0 & 2 & 2 & 1 & 0 & 0 & 9 & 8 & 7 & 7 & 2\\
0 & 2 & 1 & 3 & 4 & 8 & 6 & 0 & 0 & 2 & 0 & 1 & 1 & 1 & 2 & 0 & 0 & 9 & 8 & 7 & 7 & 1\\
0 & 2 & 1 & 3 & 5 & 7 & 5 & 0 & 0 & 3 & 0 & 0 & 2 & 2 & 1 & 0 & 0 & 9 & 8 & 7 & 7 & 1\\
0 & 3 & 0 & 0 & 7 & 11 & 3 & 0 & 0 & 0 & 2 & 1 & 2 & 0 & 2 & 0 & 0 & 9 & 8 & 7 & 7 & 1\\
0 & 3 & 0 & 2 & 6 & 6 & 6 & 0 & 0 & 3 & 0 & 0 & 3 & 1 & 1 & 0 & 0 & 9 & 8 & 7 & 7 & 1\\
0 & 3 & 0 & 4 & 1 & 12 & 2 & 0 & 2 & 2 & 0 & 0 & 2 & 1 & 1 & 1 & 0 & 10 & 5 & 10 & 6 & 1\\
0 & 3 & 0 & 4 & 2 & 10 & 2 & 0 & 3 & 2 & 0 & 0 & 2 & 0 & 3 & 0 & 0 & 10 & 5 & 10 & 6 & 1\\
0 & 3 & 1 & 12 & 3 & 6 & 0 & 0 & 0 & 0 & 0 & 3 & 0 & 0 & 3 & 0 & 0 & 12 & 4 & 6 & 9 & 1\\
1 & 0 & 0 & 0 & 0 & 0 & 0 & 2 & 0 & 25 & 0 & 0 & 0 & 0 & 0 & 3 & 0 & 5 & 10 & 15 & 1 & 1\\
1 & 0 & 0 & 0 & 0 & 0 & 25 & 3 & 0 & 0 & 0 & 0 & 0 & 0 & 0 & 1 & 1 & 5 & 15 & 5 & 6 & 1\\
1 & 0 & 1 & 0 & 0 & 6 & 8 & 0 & 2 & 9 & 0 & 0 & 1 & 0 & 3 & 0 & 0 & 7 & 9 & 11 & 4 & 1\\
1 & 0 & 2 & 0 & 7 & 5 & 10 & 1 & 0 & 0 & 0 & 2 & 1 & 2 & 0 & 0 & 0 & 8 & 11 & 4 & 8 & 1\\
1 & 0 & 2 & 2 & 4 & 10 & 4 & 0 & 0 & 3 & 0 & 0 & 2 & 3 & 0 & 0 & 0 & 9 & 8 & 7 & 7 & 2\\
1 & 1 & 1 & 1 & 4 & 10 & 6 & 0 & 0 & 2 & 0 & 1 & 2 & 1 & 1 & 0 & 0 & 9 & 8 & 7 & 7 & 2\\
1 & 2 & 0 & 2 & 1 & 14 & 2 & 0 & 2 & 2 & 0 & 0 & 3 & 1 & 0 & 1 & 0 & 10 & 5 & 10 & 6 & 2\\
1 & 2 & 0 & 2 & 2 & 12 & 2 & 0 & 3 & 2 & 0 & 0 & 3 & 0 & 2 & 0 & 0 & 10 & 5 & 10 & 6 & 1\\
1 & 3 & 0 & 9 & 5 & 6 & 0 & 0 & 0 & 1 & 0 & 2 & 3 & 0 & 1 & 0 & 0 & 12 & 4 & 6 & 9 & 1\\
2 & 0 & 0 & 0 & 0 & 0 & 0 & 0 & 25 & 0 & 0 & 0 & 0 & 0 & 0 & 4 & 0 & 10 & 0 & 20 & 1 & 1\\
2 & 0 & 0 & 0 & 0 & 25 & 0 & 1 & 0 & 0 & 0 & 0 & 0 & 0 & 0 & 2 & 1 & 10 & 5 & 10 & 6 & 2\\
2 & 0 & 0 & 0 & 25 & 0 & 0 & 2 & 0 & 0 & 0 & 0 & 0 & 0 & 0 & 0 & 2 & 10 & 10 & 0 & 11 & 2\\
2 & 1 & 1 & 8 & 3 & 10 & 0 & 0 & 0 & 0 & 0 & 3 & 2 & 0 & 1 & 0 & 0 & 12 & 4 & 6 & 9 & 1\\
2 & 2 & 0 & 7 & 5 & 8 & 0 & 0 & 0 & 1 & 0 & 2 & 4 & 0 & 0 & 0 & 0 & 12 & 4 & 6 & 9 & 1\\
3 & 0 & 0 & 25 & 0 & 0 & 0 & 0 & 0 & 0 & 0 & 0 & 0 & 0 & 0 & 1 & 2 & 15 & 0 & 5 & 11 & 1\\
3 & 0 & 0 & 6 & 6 & 12 & 0 & 0 & 0 & 0 & 0 & 0 & 3 & 0 & 0 & 0 & 1 & 12 & 4 & 6 & 9 & 1\\
\hline
\end{tabular}
\caption{Strong $(3\mod 5)$-arcs in $\PG(2,5)$ of cardinality $43$.}
\label{tab_strong_3_mod_5_arcs_pg_2_5_n_43}
\end{center}
\end{table}

\begin{table}[htp]
\setlength{\tabcolsep}{0.3mm}
\begin{center}
\begin{tabular}{|c|c|c|c|c|c|c|c|c|c|c|c|c|c|c|c|c||c|c|c|c||c|}
\hline
$A_1$ & $A_2$ & $A_3$ & $B_1$ & $B_2$ & $B_3$ & $B_4$ & $B_5$ & $B_6$ & $B_7$ & $B_8$ & $C_1$ & $C_2$ & $C_3$ & $C_4$ & $C_5$ & $D_1$ & $\lambda_0$ & $\lambda_1$ & $\lambda_2$ & $\lambda_3$ & $\#$\\
\hline
0 & 0 & 0 & 0 & 12 & 3 & 6 & 3 & 0 & 0 & 0 & 0 & 0 & 6 & 0 & 0 & 1 & 6 & 12 & 3 & 10 & 1\\
0 & 0 & 0 & 0 & 2 & 11 & 0 & 0 & 2 & 6 & 2 & 0 & 0 & 1 & 5 & 2 & 0 & 6 & 7 & 13 & 5 & 1\\
0 & 0 & 0 & 0 & 2 & 4 & 12 & 1 & 0 & 0 & 4 & 0 & 0 & 6 & 1 & 1 & 0 & 4 & 13 & 7 & 7 & 1\\
0 & 0 & 0 & 0 & 2 & 8 & 6 & 0 & 0 & 4 & 3 & 0 & 0 & 3 & 4 & 1 & 0 & 5 & 10 & 10 & 6 & 1\\
0 & 0 & 0 & 0 & 3 & 3 & 10 & 2 & 0 & 2 & 3 & 0 & 0 & 5 & 3 & 0 & 0 & 4 & 13 & 7 & 7 & 1\\
0 & 0 & 0 & 0 & 3 & 6 & 6 & 0 & 1 & 4 & 3 & 0 & 0 & 2 & 6 & 0 & 0 & 5 & 10 & 10 & 6 & 1\\
0 & 0 & 0 & 1 & 0 & 11 & 2 & 0 & 1 & 7 & 1 & 0 & 0 & 2 & 3 & 3 & 0 & 6 & 7 & 13 & 5 & 1\\
0 & 0 & 0 & 1 & 2 & 1 & 12 & 2 & 0 & 3 & 2 & 0 & 0 & 5 & 3 & 0 & 0 & 4 & 13 & 7 & 7 & 1\\
0 & 0 & 0 & 1 & 2 & 4 & 8 & 0 & 1 & 5 & 2 & 0 & 0 & 2 & 6 & 0 & 0 & 5 & 10 & 10 & 6 & 1\\
0 & 0 & 0 & 1 & 2 & 5 & 6 & 1 & 0 & 7 & 1 & 0 & 0 & 2 & 6 & 0 & 0 & 5 & 10 & 10 & 6 & 1\\
0 & 0 & 0 & 1 & 2 & 7 & 2 & 0 & 3 & 7 & 1 & 0 & 0 & 0 & 7 & 1 & 0 & 6 & 7 & 13 & 5 & 1\\
0 & 0 & 0 & 2 & 0 & 4 & 10 & 0 & 0 & 6 & 1 & 0 & 0 & 3 & 4 & 1 & 0 & 5 & 10 & 10 & 6 & 1\\
0 & 0 & 0 & 2 & 0 & 7 & 4 & 0 & 2 & 8 & 0 & 0 & 0 & 1 & 5 & 2 & 0 & 6 & 7 & 13 & 5 & 1\\
0 & 0 & 0 & 2 & 0 & 8 & 0 & 0 & 7 & 6 & 0 & 0 & 0 & 0 & 4 & 4 & 0 & 7 & 4 & 16 & 4 & 1\\
0 & 0 & 0 & 3 & 6 & 6 & 9 & 0 & 0 & 0 & 0 & 0 & 0 & 3 & 3 & 0 & 1 & 7 & 9 & 6 & 9 & 1\\
0 & 0 & 0 & 4 & 2 & 14 & 4 & 0 & 0 & 0 & 0 & 0 & 0 & 1 & 4 & 1 & 1 & 8 & 6 & 9 & 8 & 1\\
0 & 0 & 1 & 0 & 9 & 3 & 6 & 3 & 0 & 0 & 0 & 3 & 0 & 6 & 0 & 0 & 0 & 6 & 12 & 3 & 10 & 1\\
0 & 0 & 1 & 1 & 8 & 7 & 2 & 0 & 0 & 1 & 2 & 0 & 3 & 6 & 0 & 0 & 0 & 7 & 9 & 6 & 9 & 1\\
0 & 0 & 1 & 2 & 5 & 7 & 6 & 0 & 0 & 0 & 1 & 2 & 1 & 4 & 2 & 0 & 0 & 7 & 9 & 6 & 9 & 1\\
0 & 0 & 1 & 3 & 4 & 5 & 8 & 0 & 0 & 1 & 0 & 2 & 1 & 4 & 2 & 0 & 0 & 7 & 9 & 6 & 9 & 1\\
0 & 0 & 1 & 3 & 5 & 4 & 7 & 0 & 0 & 2 & 0 & 1 & 2 & 5 & 1 & 0 & 0 & 7 & 9 & 6 & 9 & 2\\
0 & 0 & 1 & 3 & 6 & 3 & 6 & 0 & 0 & 3 & 0 & 0 & 3 & 6 & 0 & 0 & 0 & 7 & 9 & 6 & 9 & 2\\
0 & 0 & 1 & 4 & 2 & 10 & 2 & 0 & 1 & 2 & 0 & 1 & 2 & 2 & 4 & 0 & 0 & 8 & 6 & 9 & 8 & 2\\
0 & 0 & 1 & 4 & 3 & 9 & 1 & 0 & 1 & 3 & 0 & 0 & 3 & 3 & 3 & 0 & 0 & 8 & 6 & 9 & 8 & 2\\
0 & 0 & 1 & 6 & 0 & 9 & 0 & 0 & 6 & 0 & 0 & 0 & 3 & 0 & 6 & 0 & 0 & 9 & 3 & 12 & 7 & 1\\
0 & 0 & 2 & 6 & 12 & 0 & 0 & 0 & 0 & 0 & 1 & 6 & 0 & 4 & 0 & 0 & 0 & 9 & 8 & 2 & 12 & 1\\
0 & 1 & 0 & 1 & 6 & 6 & 7 & 0 & 0 & 0 & 1 & 2 & 2 & 3 & 2 & 0 & 0 & 7 & 9 & 6 & 9 & 1\\
0 & 1 & 0 & 1 & 7 & 5 & 6 & 0 & 0 & 1 & 1 & 1 & 3 & 4 & 1 & 0 & 0 & 7 & 9 & 6 & 9 & 1\\
0 & 1 & 0 & 2 & 3 & 13 & 0 & 1 & 0 & 2 & 0 & 2 & 2 & 0 & 5 & 0 & 0 & 8 & 6 & 9 & 8 & 1\\
0 & 1 & 0 & 2 & 5 & 4 & 9 & 0 & 0 & 1 & 0 & 2 & 2 & 3 & 2 & 0 & 0 & 7 & 9 & 6 & 9 & 1\\
0 & 1 & 0 & 3 & 1 & 12 & 4 & 0 & 0 & 1 & 0 & 2 & 2 & 1 & 3 & 1 & 0 & 8 & 6 & 9 & 8 & 1\\
0 & 1 & 0 & 3 & 2 & 11 & 3 & 0 & 0 & 2 & 0 & 1 & 3 & 2 & 2 & 1 & 0 & 8 & 6 & 9 & 8 & 1\\
0 & 1 & 0 & 3 & 3 & 10 & 2 & 0 & 0 & 3 & 0 & 0 & 4 & 3 & 1 & 1 & 0 & 8 & 6 & 9 & 8 & 1\\
0 & 1 & 0 & 3 & 3 & 9 & 3 & 0 & 1 & 2 & 0 & 1 & 3 & 1 & 4 & 0 & 0 & 8 & 6 & 9 & 8 & 2\\
0 & 1 & 0 & 3 & 4 & 8 & 2 & 0 & 1 & 3 & 0 & 0 & 4 & 2 & 3 & 0 & 0 & 8 & 6 & 9 & 8 & 2\\
0 & 1 & 0 & 5 & 0 & 10 & 1 & 0 & 5 & 0 & 0 & 0 & 4 & 0 & 4 & 1 & 0 & 9 & 3 & 12 & 7 & 1\\
0 & 2 & 0 & 12 & 0 & 6 & 0 & 0 & 1 & 0 & 0 & 2 & 6 & 1 & 0 & 1 & 0 & 11 & 2 & 8 & 10 & 1\\
1 & 0 & 0 & 0 & 6 & 5 & 8 & 0 & 0 & 2 & 0 & 1 & 4 & 4 & 0 & 0 & 0 & 7 & 9 & 6 & 9 & 1\\
1 & 0 & 0 & 1 & 4 & 10 & 2 & 0 & 1 & 3 & 0 & 0 & 5 & 2 & 2 & 0 & 0 & 8 & 6 & 9 & 8 & 1\\
1 & 0 & 0 & 4 & 14 & 0 & 4 & 0 & 0 & 0 & 0 & 4 & 1 & 2 & 0 & 0 & 1 & 9 & 8 & 2 & 12 & 1\\
1 & 0 & 1 & 4 & 11 & 0 & 4 & 0 & 0 & 0 & 0 & 7 & 1 & 2 & 0 & 0 & 0 & 9 & 8 & 2 & 12 & 1\\
2 & 0 & 0 & 8 & 1 & 8 & 0 & 0 & 2 & 0 & 0 & 2 & 8 & 0 & 0 & 0 & 0 & 11 & 2 & 8 & 10 & 2\\
\hline
\end{tabular}
\caption{Strong $(3\mod 5)$-arcs in $\PG(2,5)$ of cardinality $48$.}
\label{tab_strong_3_mod_5_arcs_pg_2_5_n_48}
\end{center}
\end{table}

\begin{table}[htp]
\setlength{\tabcolsep}{0.3mm}
\begin{center}
\begin{tabular}{|c|c|c|c|c|c|c|c|c|c|c|c|c|c|c|c|c||c|c|c|c||c|}
\hline
$A_1$ & $A_2$ & $A_3$ & $B_1$ & $B_2$ & $B_3$ & $B_4$ & $B_5$ & $B_6$ & $B_7$ & $B_8$ & $C_1$ & $C_2$ & $C_3$ & $C_4$ & $C_5$ & $D_1$ & $\lambda_0$ & $\lambda_1$ & $\lambda_2$ & $\lambda_3$ & $\#$\\
\hline
0 & 0 & 0 & 0 & 5 & 10 & 0 & 0 & 0 & 0 & 2 & 2 & 4 & 4 & 4 & 0 & 0 & 6 & 7 & 8 & 10 & 1\\
0 & 0 & 0 & 0 & 6 & 4 & 5 & 1 & 0 & 0 & 1 & 3 & 2 & 8 & 1 & 0 & 0 & 5 & 10 & 5 & 11 & 1\\
0 & 0 & 0 & 1 & 3 & 9 & 3 & 0 & 0 & 0 & 1 & 3 & 3 & 3 & 5 & 0 & 0 & 6 & 7 & 8 & 10 & 1\\
0 & 0 & 0 & 1 & 4 & 3 & 8 & 1 & 0 & 0 & 0 & 4 & 1 & 7 & 2 & 0 & 0 & 5 & 10 & 5 & 11 & 1\\
0 & 0 & 0 & 1 & 6 & 0 & 8 & 0 & 1 & 0 & 1 & 2 & 3 & 9 & 0 & 0 & 0 & 5 & 10 & 5 & 11 & 1\\
0 & 0 & 0 & 2 & 2 & 7 & 5 & 0 & 0 & 1 & 0 & 3 & 3 & 3 & 5 & 0 & 0 & 6 & 7 & 8 & 10 & 1\\
0 & 0 & 0 & 2 & 3 & 6 & 4 & 0 & 0 & 2 & 0 & 2 & 4 & 4 & 4 & 0 & 0 & 6 & 7 & 8 & 10 & 2\\
0 & 0 & 0 & 3 & 0 & 11 & 1 & 0 & 1 & 1 & 0 & 1 & 6 & 2 & 3 & 2 & 0 & 7 & 4 & 11 & 9 & 1\\
0 & 0 & 0 & 3 & 2 & 8 & 0 & 0 & 2 & 2 & 0 & 0 & 7 & 2 & 4 & 1 & 0 & 7 & 4 & 11 & 9 & 2\\
0 & 0 & 0 & 9 & 3 & 6 & 0 & 0 & 0 & 0 & 0 & 3 & 6 & 0 & 3 & 0 & 1 & 9 & 3 & 7 & 12 & 1\\
0 & 0 & 1 & 9 & 0 & 6 & 0 & 0 & 0 & 0 & 0 & 6 & 6 & 0 & 3 & 0 & 0 & 9 & 3 & 7 & 12 & 1\\
0 & 1 & 0 & 1 & 12 & 0 & 0 & 2 & 0 & 0 & 0 & 11 & 0 & 4 & 0 & 0 & 0 & 7 & 9 & 1 & 14 & 1\\
0 & 1 & 0 & 8 & 2 & 4 & 0 & 0 & 0 & 1 & 0 & 5 & 8 & 0 & 2 & 0 & 0 & 9 & 3 & 7 & 12 & 1\\
1 & 0 & 0 & 10 & 0 & 0 & 0 & 0 & 5 & 0 & 0 & 0 & 15 & 0 & 0 & 0 & 0 & 10 & 0 & 10 & 11 & 1\\
1 & 0 & 0 & 6 & 2 & 6 & 0 & 0 & 0 & 1 & 0 & 5 & 9 & 0 & 1 & 0 & 0 & 9 & 3 & 7 & 12 & 1\\
\hline
\end{tabular}
\caption{Strong $(3\mod 5)$-arcs in $\PG(2,5)$ of cardinality $53$.}
\label{tab_strong_3_mod_5_arcs_pg_2_5_n_53}
\end{center}
\end{table}

\begin{table}[htp]
\setlength{\tabcolsep}{0.3mm}
\begin{center}
\begin{tabular}{|c|c|c|c|c|c|c|c|c|c|c|c|c|c|c|c|c||c|c|c|c||c|}
\hline
$A_1$ & $A_2$ & $A_3$ & $B_1$ & $B_2$ & $B_3$ & $B_4$ & $B_5$ & $B_6$ & $B_7$ & $B_8$ & $C_1$ & $C_2$ & $C_3$ & $C_4$ & $C_5$ & $D_1$ & $\lambda_0$ & $\lambda_1$ & $\lambda_2$ & $\lambda_3$ & $\#$\\
\hline
0 & 0 & 0 & 0 & 0 & 10 & 0 & 1 & 0 & 0 & 0 & 5 & 5 & 0 & 10 & 0 & 0 & 5 & 5 & 10 & 11 & 1\\
0 & 0 & 0 & 0 & 3 & 3 & 3 & 0 & 0 & 1 & 1 & 3 & 5 & 9 & 3 & 0 & 0 & 4 & 8 & 7 & 12 & 1\\
0 & 0 & 0 & 1 & 1 & 2 & 6 & 0 & 0 & 1 & 0 & 4 & 4 & 8 & 4 & 0 & 0 & 4 & 8 & 7 & 12 & 1\\
0 & 0 & 0 & 1 & 1 & 5 & 2 & 0 & 1 & 1 & 0 & 3 & 7 & 2 & 8 & 0 & 0 & 5 & 5 & 10 & 11 & 1\\
0 & 0 & 0 & 1 & 1 & 6 & 1 & 0 & 0 & 2 & 0 & 2 & 8 & 4 & 5 & 1 & 0 & 5 & 5 & 10 & 11 & 1\\
0 & 0 & 0 & 1 & 2 & 4 & 1 & 0 & 1 & 2 & 0 & 2 & 8 & 3 & 7 & 0 & 0 & 5 & 5 & 10 & 11 & 1\\
0 & 0 & 0 & 1 & 3 & 0 & 4 & 0 & 0 & 3 & 0 & 2 & 6 & 10 & 2 & 0 & 0 & 4 & 8 & 7 & 12 & 1\\
0 & 0 & 0 & 1 & 3 & 3 & 0 & 0 & 1 & 3 & 0 & 1 & 9 & 4 & 6 & 0 & 0 & 5 & 5 & 10 & 11 & 1\\
0 & 0 & 0 & 2 & 1 & 4 & 0 & 0 & 4 & 0 & 0 & 0 & 12 & 0 & 6 & 2 & 0 & 6 & 2 & 13 & 10 & 1\\
0 & 0 & 0 & 3 & 6 & 0 & 3 & 0 & 0 & 0 & 0 & 9 & 3 & 6 & 0 & 0 & 1 & 6 & 7 & 3 & 15 & 1\\
0 & 0 & 1 & 3 & 3 & 0 & 3 & 0 & 0 & 0 & 0 & 12 & 3 & 6 & 0 & 0 & 0 & 6 & 7 & 3 & 15 & 1\\
\hline
\end{tabular}
\caption{Strong $(3\mod 5)$-arcs in $\PG(2,5)$ of cardinality $58$.}
\label{tab_strong_3_mod_5_arcs_pg_2_5_n_58}
\end{center}
\end{table}

\begin{table}[htp]
\setlength{\tabcolsep}{0.3mm}
\begin{center}
\begin{tabular}{|c|c|c|c|c|c|c|c|c|c|c|c|c|c|c|c|c||c|c|c|c||c|}
\hline
$A_1$ & $A_2$ & $A_3$ & $B_1$ & $B_2$ & $B_3$ & $B_4$ & $B_5$ & $B_6$ & $B_7$ & $B_8$ & $C_1$ & $C_2$ & $C_3$ & $C_4$ & $C_5$ & $D_1$ & $\lambda_0$ & $\lambda_1$ & $\lambda_2$ & $\lambda_3$ & $\#$\\
\hline
0 & 0 & 0 & 0 & 0 & 0 & 0 & 0 & 0 & 0 & 5 & 0 & 0 & 15 & 10 & 1 & 0 & 0 & 10 & 10 & 11 & 1\\
0 & 0 & 0 & 0 & 0 & 0 & 0 & 0 & 0 & 3 & 2 & 1 & 2 & 6 & 15 & 2 & 0 & 1 & 7 & 13 & 10 & 1\\
0 & 0 & 0 & 0 & 0 & 0 & 0 & 0 & 1 & 4 & 0 & 0 & 6 & 2 & 12 & 6 & 0 & 2 & 4 & 16 & 9 & 1\\
0 & 0 & 0 & 0 & 0 & 1 & 4 & 1 & 0 & 0 & 0 & 4 & 2 & 14 & 4 & 0 & 1 & 2 & 9 & 6 & 14 & 1\\
0 & 0 & 0 & 0 & 0 & 3 & 3 & 0 & 0 & 0 & 0 & 3 & 6 & 6 & 9 & 0 & 1 & 3 & 6 & 9 & 13 & 1\\
0 & 0 & 0 & 0 & 0 & 6 & 0 & 0 & 0 & 0 & 0 & 0 & 12 & 3 & 6 & 3 & 1 & 4 & 3 & 12 & 12 & 1\\
0 & 0 & 0 & 4 & 1 & 2 & 0 & 0 & 0 & 0 & 0 & 4 & 14 & 0 & 4 & 0 & 2 & 6 & 2 & 8 & 15 & 1\\
0 & 0 & 1 & 0 & 0 & 0 & 0 & 0 & 0 & 3 & 0 & 3 & 9 & 9 & 6 & 0 & 0 & 3 & 6 & 9 & 13 & 1\\
0 & 0 & 1 & 0 & 0 & 0 & 0 & 0 & 3 & 0 & 0 & 3 & 12 & 0 & 12 & 0 & 0 & 4 & 3 & 12 & 12 & 1\\
\hline
\end{tabular}
\caption{Strong $(3\mod 5)$-arcs in $\PG(2,5)$ of cardinality $63$.}
\label{tab_strong_3_mod_5_arcs_pg_2_5_n_63}
\end{center}
\end{table}

\begin{table}[htp]
\setlength{\tabcolsep}{0.3mm}
\begin{center}
\begin{tabular}{|c|c|c|c|c|c|c|c|c|c|c|c|c|c|c|c|c||c|c|c|c||c|}
\hline
$A_1$ & $A_2$ & $A_3$ & $B_1$ & $B_2$ & $B_3$ & $B_4$ & $B_5$ & $B_6$ & $B_7$ & $B_8$ & $C_1$ & $C_2$ & $C_3$ & $C_4$ & $C_5$ & $D_1$ & $\lambda_0$ & $\lambda_1$ & $\lambda_2$ & $\lambda_3$ & $\#$\\
\hline
0 & 0 & 0 & 0 & 0 & 0 & 0 & 0 & 0 & 0 & 0 & 0 & 0 & 0 & 0 & 30 & 1 & 0 & 0 & 25 & 6 & 1\\
0 & 0 & 0 & 0 & 0 & 0 & 0 & 1 & 0 & 0 & 0 & 0 & 0 & 0 & 25 & 3 & 2 & 0 & 5 & 15 & 11 & 1\\
0 & 0 & 0 & 0 & 0 & 0 & 0 & 2 & 0 & 0 & 0 & 0 & 0 & 25 & 0 & 1 & 3 & 0 & 10 & 5 & 16 & 1\\
0 & 0 & 0 & 0 & 3 & 0 & 0 & 0 & 0 & 0 & 0 & 6 & 6 & 12 & 0 & 0 & 4 & 3 & 6 & 4 & 18 & 1\\
1 & 0 & 0 & 0 & 0 & 0 & 0 & 0 & 0 & 0 & 0 & 0 & 25 & 0 & 0 & 2 & 3 & 5 & 0 & 10 & 16 & 1\\
1 & 0 & 0 & 0 & 0 & 0 & 0 & 1 & 0 & 0 & 0 & 25 & 0 & 0 & 0 & 0 & 4 & 5 & 5 & 0 & 21 & 1\\
\hline
\end{tabular}
\caption{Strong $(3\mod 5)$-arcs in $\PG(2,5)$ of cardinality $68$.}
\label{tab_strong_3_mod_5_arcs_pg_2_5_n_68}
\end{center}
\end{table}

\begin{table}[htp]
\setlength{\tabcolsep}{0.3mm}
\begin{center}
\begin{tabular}{|c|c|c|c|c|c|c|c|c|c|c|c|c|c|c|c|c||c|c|c|c||c|}
\hline
$A_1$ & $A_2$ & $A_3$ & $B_1$ & $B_2$ & $B_3$ & $B_4$ & $B_5$ & $B_6$ & $B_7$ & $B_8$ & $C_1$ & $C_2$ & $C_3$ & $C_4$ & $C_5$ & $D_1$ & $\lambda_0$ & $\lambda_1$ & $\lambda_2$ & $\lambda_3$ & $\#$\\
\hline
0 & 0 & 0 & 0 & 0 & 0 & 0 & 0 & 0 & 0 & 0 & 0 & 0 & 0 & 0 & 0 & 31 & 0 & 0 & 0 & 31 & 1\\
\hline
\end{tabular}
\caption{Strong $(3\mod 5)$-arcs in $\PG(2,5)$ of cardinality $93$.}
\label{tab_strong_3_mod_5_arcs_pg_2_5_n_93}
\end{center}
\end{table}

\end{document}